\documentclass{article}

\usepackage[english]{babel}
\usepackage[T1]{fontenc}
\usepackage{graphicx}
\usepackage{makeidx}
\usepackage{hyperref}
\usepackage{amsmath,amssymb,amsthm}
\usepackage[top=1in, bottom=1in, left=1in, right=1in]{geometry}
\usepackage{natbib}

\usepackage{setspace}
\newcommand{\blind}{0}

\newtheorem{theorem}{Theorem}
\newtheorem{proposition}{Proposition}
\theoremstyle{definition}
\newtheorem*{definition}{Definition}
\theoremstyle{remark}
\newtheorem*{remark}{Remark}

\def \a{\alpha}
\def \d{\delta}
\def \e{\epsilon}

\def \g{\gamma}
\def \s{\sigma}

\def \R{\mathbb{R}}
\def \P{\mathbb{P}}
\def \E{\mathbb{E}}
\def \Ri{\mathcal{R}}
\def \Mi{\mathcal{M}}
\def \1{{\bf{1}}}
\def \b{\boldsymbol}

\title{A Random Field Model and its Application in \\ Industrial Production.}
\if0\blind
{
\author{Julie OGER$^{\ast\dag}$, Emmanuel LESIGNE$^\ast$, Philippe LEDUC$^\dag$ \\
$^\ast$ Université François-Rabelais, CNRS, LMPT UMR7350, Tours, France \\
$^\dag$ STMicroelectronics, Tours, France}
} \else{
\author{}
}\fi
\date{}

\begin{document}
\maketitle

%%%%%%%%%%%%%%%%%%%%%%%%%%%%%%%%%%%%%%%%%%%%%%%%%%%%%%%%%%%%%%%%%%%%%%%%%%%%%%%%%%%%%%%%%%%%%%%%%%%%%%%%%%%%%%%%%%%%%%%%%%%%%%%%%%%%%%%%%
% Abstract
%%%%%%%%%%%%%%%%%%%%%%%%%%%%%%%%%%%%%%%%%%%%%%%%%%%%%%%%%%%%%%%%%%%%%%%%%%%%%%%%%%%%%%%%%%%%%%%%%%%%%%%%%%%%%%%%%%%%%%%%%%%%%%%%%%%%%%%%%
\begin{abstract}
In competitive industries, a reliable yield forecasting is a prime factor to accurately determine the production costs and therefore ensure profitability.
Indeed, quantifying the risks long before the effective manufacturing process enables fact-based decision-making.
From the development stage, improvement efforts can be early identified and prioritized.
In order to measure the impact of industrial process fluctuations on the product performances, the construction of a failure risk probability estimator is presented in this article.
The complex relationship between the process technology and the product design (non linearities, multi-modal features...) is handled via random process regression.
A random field encodes, for each product configuration, the available information regarding the risk of non-compliance.
After a brief presentation of the Gaussian model approach, we describe a Bayesian reasoning avoiding a priori choices of location and scale parameters.
The Gaussian mixture prior, conditioned by measured (or calculated) data, yields a posterior characterized by a multivariate Student distribution.
The probabilistic nature of the model is then operated to derive a failure risk probability, defined as a random variable.
To do this, our approach is to consider as random all unknown, inaccessible or fluctuating data.
In order to propagate uncertainties, a fuzzy set approach provides an appropriate framework for the implementation of a Bayesian model mimicking expert elicitation.
The underlying leitmotiv is to insert minimal a priori information in the failure risk model.
The relevancy of this concept is illustrated with theoretical examples.
Note that this article comes with supplementary material available on line.
\end{abstract}

%%%%%%%%%%%%%%%%%%%%%%%%%%%%%%%%%%%%%%%%%%%%%%%%%%%%%%%%%%%%%%%%%%%%%%%%%%%%%%%%%%%%%%%%%%%%%%%%%%%%%%%%%%%%%%%%%%%%%%%%%%%%%%%%%%%%%%%%%
% Keywords
%%%%%%%%%%%%%%%%%%%%%%%%%%%%%%%%%%%%%%%%%%%%%%%%%%%%%%%%%%%%%%%%%%%%%%%%%%%%%%%%%%%%%%%%%%%%%%%%%%%%%%%%%%%%%%%%%%%%%%%%%%%%%%%%%%%%%%%%%
{\bf Keywords}: Kriging, Bayesian inference, Gaussian processes mixture prior, multivariate t-distribution, uncertainty analysis, manufacturing yield evaluation

%%%%%%%%%%%%%%%%%%%%%%%%%%%%%%%%%%%%%%%%%%%%%%%%%%%%%%%%%%%%%%%%%%%%%%%%%%%%%%%%%%%%%%%%%%%%%%%%%%%%%%%%%%%%%%%%%%%%%%%%%%%%%%%%%%%%%%%%%
% Introduction
%%%%%%%%%%%%%%%%%%%%%%%%%%%%%%%%%%%%%%%%%%%%%%%%%%%%%%%%%%%%%%%%%%%%%%%%%%%%%%%%%%%%%%%%%%%%%%%%%%%%%%%%%%%%%%%%%%%%%%%%%%%%%%%%%%%%%%%%%
\section{Introduction}

%%%%%%%%%%%%%%%%%%%%%%%%%%%%%%%%%%%%%%%%%%%%%%%%%%%%%%%%%%%%%%%%%%%%%%%%%%%%%%%%%%%%%%%%%%%%%%%%%%%%%%%%%%%%%%%%%%%%%%%%%%%%%%%%%%%%%%%%%
% Settings and motivations
%%%%%%%%%%%%%%%%%%%%%%%%%%%%%%%%%%%%%%%%%%%%%%%%%%%%%%%%%%%%%%%%%%%%%%%%%%%%%%%%%%%%%%%%%%%%%%%%%%%%%%%%%%%%%%%%%%%%%%%%%%%%%%%%%%%%%%%%%
\subsection{Settings and motivations}
During the fabrication of an industrial product, the manufacturing process cannot be entirely controlled.
There is an intrinsic variability (materials, machinery, environment ...) which can significantly deteriorate the characteristics of the manufactured pieces, inducing non-functional parts.
Such variations are critical when dealing with complex technical systems, such as for example integrated circuits developed in the microelectronic industry.
Indeed, if this fact is covered up, the resulting circuits are (in many cases) either consistently out of specifications or unnecessarily overdesigned.

To investigate the influence of fluctuating parameters, one can perform a corner analysis applying Design of Experiments principles.
However, the conclusions drawn from such measurements campaigns are, in general, qualitative and consequently incomplete.
Besides, time and money required by physical prototyping are often prohibitive.
As an alternative, numerical (deterministic) models implemented in engineering simulation software offer a way to compute the relevant features (thermal, mechanical, electrical, ...) of a device.
Thus, nowadays, engineers can virtually explore various design configurations and get a deep insight of the final product performances.
In this context, statistical studies can be conducted to evaluate the effect of process tolerances on the manufacturing yield.
The most popular is the Monte Carlo (MC) method, which consists in sampling configurations of the product parameters according to their probability distributions and to count the failure events.
The process is monitored "in line" to roughly assess distributions of environmental variables.
This method is easy to implement but its efficiency depends directly on the complexity of the deterministic model.
When the numerical simulator considered is time and/or memory demanding, we only get partial information.
For instance, the duration of "Finite Elements" analyses, from several hours to several days, is not compatible with a brute Monte Carlo approach.

Data are sparse and it is therefore necessary to propagate the information in the factor space using an analytical representation (emulator).
Monte Carlo sampling is then applied to this analytical representation, a fast surrogate of the computer code (simulator), see \cite{Pfingsten}.
In their article, outputs are considered as a single outcome of a Gaussian random field.
Our work focuses on a mixture of Gaussian fields, and this choice will be justified in the sequel.
Classical modeling approaches for failure risk (or yield) estimation only keep a small part of the model available information: mean, quantile...
Besides, after the model has passed the validation tests, these outputs are taken at face values.
We believe such a procedure is hazardous in the specific field of risk assessment.
Indeed, it does not measure the impact of the uncertainty introduced by any modeling stage on the only quantity of interest, namely the risk of failure.
For example, how can the decision-maker relate the model acceptance criteria to the accuracy of the failure risk estimate?
Answering this question in particular is difficult and probably fruitless.
Our work proposes a novel and general solution to address these issues.
Once the random model has been determined, and given the probability distribution of the product parameters, we go beyond Pfingsten's approach to define the failure risk probability.
The predictive uncertainty is not deduced from the posterior model, calculating for instance the conditional variance.
The failure risk itself is probabilistic and randomness mirrors the model uncertainty.

We consider that the product (meaning each individual manufactured piece) under study is characterized by a number $D$ of numerical factors.
Each factor can vary in a given interval what allows the definition of the factor space $ X \subseteq \R^D $. 
Each set of factors $x$ ($\in X$) determines a numerical value $ y \left( x \right) $, and the specifications imposed on the product apply to the value of this response $y$.
We can derive from these specifications the \textit{out of specifications space} $ A \subseteq \R $: the product characterized by the set of factors $x$ does not satisfy the specifications if and only if $ y \left( x \right) \in A $.
The factor space $X$ and the out of specifications space $A$ are considered as known, but regarding the deterministic function $ x \mapsto y \left( x \right) $, we have only very partial information.

Besides, a probability distribution $P$ is given on the factor space $X$.
This distribution reflects the factors variability and is considered as known.

As we said previously, the knowledge of $ y \left( x \right) $ is rarely available for all $x$.
We only have access to a restraint number of data.
Suppose we know $n$ deterministic response values $ \left( y_i = y \left( x_i \right) \right)_{1 \leq i \leq n} $ respectively for factor set values $ \left( x_i \right)_{ 1 \leq i \leq n } $.
With these data, our goal is to define a random variable named \textit{failure risk probability}.
Its distribution should help for the robust estimation of the product manufacturability.
 
Let us begin with the construction of a random field model $ \left( Y_x \right)_{x \in X} $ of the unknown response $ \left( y \left( x \right) \right)_{x \in X} $.

%%%%%%%%%%%%%%%%%%%%%%%%%%%%%%%%%%%%%%%%%%%%%%%%%%%%%%%%%%%%%%%%%%%%%%%%%%%%%%%%%%%%%%%%%%%%%%%%%%%%%%%%%%%%%%%%%%%%%%%%%%%%%%%%%%%%%%%%%
% Model construction
%%%%%%%%%%%%%%%%%%%%%%%%%%%%%%%%%%%%%%%%%%%%%%%%%%%%%%%%%%%%%%%%%%%%%%%%%%%%%%%%%%%%%%%%%%%%%%%%%%%%%%%%%%%%%%%%%%%%%%%%%%%%%%%%%%%%%%%%%
\subsection{Model construction}
Several methods are described in the literature to infer a model $ \left( Y_x \right)_{x \in X} $ from a limited number of available deterministic data
\[ Y_{x_1} = y_1, Y_{x_2} = y_2, \ldots, Y_{x_n} = y_n. \]
Linear regression analysis is the method of choice in the scientific community for estimating relationships among predictor variables.
When the phenomenon studied is complex though, the order of the model (usually polynomial), necessary to correctly fit the data, leads to numerous unknowns.
This drawback is magnified in a high-dimensional factor space.
Indeed, for a $p$-variable polynomial of degree $d$, $ \begin{pmatrix} p+d\\d\end{pmatrix} $ coefficients should be determined.
If the number of data available is less than this limit, the model is singular.
Thus, the player is required to make a "blind bet" to enter the game...
In addition, the assumptions justifying the statistical model are rarely strictly respected in practice.
For example, how should we understand random errors uncorrelated with zero mean when data are obtained through computer experiments, reproducible in essence?

These difficulties are overcome by the Gaussian Process (GP) model.
The original kriging method was formalized by \cite{Matheron} in the geo-sciences.
First, a GP indexed by $X$ is selected and then conditioned by the data.
This approach has been widely used for different applications such as, geostatistics studies (\cite{Berger}), optimization (\cite{Emmerich}), wind fields modeling (\cite{CornFord}) or design sensitivity analysis (\cite{Pfingsten}), justified by better predictive performances than several other regression methods (\cite{Rasmussen}). 
Indeed, GP modeling has several interesting properties.
Belonging to the interpolation methods family, this model exactly reproduces the observed data set, there is no residual at the observation points.
As a consequence, it is an appropriate tool for the analysis of computer experiments (\cite{Sacks}).
GP can be determined (in theory) even for data set of small size, a key property when information is missing which is a common situation when the number of factors ($D$) is large.
Moreover, it is a very versatile model, able to describe non-continuous as well as non-differentiable surface responses.
This noteworthy feature is particularly useful to handle response discontinuities, which may occur due to the numerical solving scheme of computer codes (especially when meshing algorithms are involved).
Finally, the probabilistic nature of the predictions can be interpreted as a model for uncertainty, a confidence interval representing a degree of belief.

A major and often neglected problem of this method is the a priori choice of parameters (mean, variance and correlation function), left to the responsibility of the data analyst.
In order to reduce the arbitrary nature of expert elicitation, we propose to replace the Gaussian field by a mixture of Gaussian fields, therefore avoiding the a priori choice of mean and variance.
For this purpose, with the aim of introducing minimal a priori information into the model, mean and variance are considered as random variables with uniform probability distributions.
We will show that, in doing so, any posterior random vector follows a multivariate t-distribution (or {\it multivariate Student distribution}).
We note that a similar result has been obtained with a hierarchical Bayesian model approach described, for example, in \cite{Santner}.

%%%%%%%%%%%%%%%%%%%%%%%%%%%%%%%%%%%%%%%%%%%%%%%%%%%%%%%%%%%%%%%%%%%%%%%%%%%%%%%%%%%%%%%%%%%%%%%%%%%%%%%%%%%%%%%%%%%%%%%%%%%%%%%%%%%%%%%%%
% Failure risk probability viewed as a random variable
%%%%%%%%%%%%%%%%%%%%%%%%%%%%%%%%%%%%%%%%%%%%%%%%%%%%%%%%%%%%%%%%%%%%%%%%%%%%%%%%%%%%%%%%%%%%%%%%%%%%%%%%%%%%%%%%%%%%%%%%%%%%%%%%%%%%%%%%%
\subsection{Failure risk probability viewed as a random variable}
\label{Intro_FailureRisk}
The random field $ \left( Y_x \right)_{x \in X} $ represents the information as well as the uncertainty regarding the values of the response function $y$.
Once this field is built, we can compute for each factor point $ x \in X $, the \textit{failure probability} $ \P \left( Y_x \in A \right) $.

The reader sees here that we denote by $\P$ the probability measure associated to the random field (meaning that the random field is defined on an abstract probability space $ \left( \Omega, \P \right) $).
On the other hand, we recall that we denote by $P$ the probability measure assigned to the factor space $X$.

We decide that there is a \textit{failure risk} when the failure probability exceeds a fixed threshold $ \a \in \left[ 0, 1 \right] $.
So, \textit{the failure risk probability} (in short, \textit{risk probability}) is defined by, $ R \left( \a \right) := P \left( \P \left( Y_x \in A \right) > \a \right) $,
that is, $ R \left( \a \right) = P \left( \left\{ x \in X \mid \P \left( Y_x \in A \right) > \a \right\} \right)$. 
For a fixed $\a$, the distribution of the risk probability is a Dirac mass at the point $ R \left( \a \right) $.
This Dirac mass is denoted by $ \d_{R \left( \a \right) } $.

Setting the correct value for the accident threshold $\a$ is a tricky problem.
On the one hand, the decision will significantly impact the risk assessment of the product.
On the other hand, this choice is eminently subjective since it depends heavily on the risk attitude of the individual.
That is why we consider $\a$ as a random variable.
As a consequence, $ R \left( \a \right) $ is also a random variable and the risk probability has the distribution
\[ \Ri := \int_0^1 \d_{R \left( \a \right) } \text{ }\eta \left( \text{d} \a \right) \]
where $\eta$ is the probability distribution of $\a$.
We will see hereafter that the uniform distribution is a choice for $\eta$ which provides an interesting property of mean value preservation.

In practice, $ R \left( \a \right) $ cannot be computed analytically, it is approximated via a Monte Carlo simulation.
According to standard Bayesian methods for sampling study, $ R \left( \a \right) $ follows a beta distribution, denoted here by $ \beta_{\a} $.
Finally, the distribution of the risk probability is defined, if $\eta$ has been chosen uniform, by
\[ \Ri := \int_0^1 \beta_{\a} \text{ d} \a .\]
This reasoning will be described with more details in the sequel.

%%%%%%%%%%%%%%%%%%%%%%%%%%%%%%%%%%%%%%%%%%%%%%%%%%%%%%%%%%%%%%%%%%%%%%%%%%%%%%%%%%%%%%%%%%%%%%%%%%%%%%%%%%%%%%%%%%%%%%%%%%%%%%%%%%%%%%%%%
% Contents
%%%%%%%%%%%%%%%%%%%%%%%%%%%%%%%%%%%%%%%%%%%%%%%%%%%%%%%%%%%%%%%%%%%%%%%%%%%%%%%%%%%%%%%%%%%%%%%%%%%%%%%%%%%%%%%%%%%%%%%%%%%%%%%%%%%%%%%%%
\subsection{Contents}
In Section~\ref{RandomModels}, several models are discussed. After a reminder about Gaussian fields (\ref{GaussianField}), we show that a random field prior with unknown mean (\ref{GaussianMixture}) and variance leads to a conditioned random field following a multivariate t-distribution (\ref{StudentField}), and discuss on the model implementation (\ref{ModelImplementation}).
In Section~\ref{Applications}, the construction of the density of the failure risk probability is described, beginning with the elicitation model (\ref{FailureRisk}).
We provide arguments in favor of a risk-neutral attitude, that is a uniform distribution for the accident threshold (\ref{UniformDistribution}).
We continue with practical considerations (\ref{MCapproach}) to conclude with the description of our global strategy (\ref{GlobalStrategy}).
Each section ends with illustrating examples.

%%%%%%%%%%%%%%%%%%%%%%%%%%%%%%%%%%%%%%%%%%%%%%%%%%%%%%%%%%%%%%%%%%%%%%%%%%%%%%%%%%%%%%%%%%%%%%%%%%%%%%%%%%%%%%%%%%%%%%%%%%%%%%%%%%%%%%%%%
% From Gaussian to Student fields
%%%%%%%%%%%%%%%%%%%%%%%%%%%%%%%%%%%%%%%%%%%%%%%%%%%%%%%%%%%%%%%%%%%%%%%%%%%%%%%%%%%%%%%%%%%%%%%%%%%%%%%%%%%%%%%%%%%%%%%%%%%%%%%%%%%%%%%%%
\section{From Gaussian to Student fields}
\label{RandomModels}
The unknown function $ x \mapsto y \left( x \right) $ is modeled as a random function.
In the absolute sense, if we disregard the variability that is sometimes introduced by numerical solving schemes, the function is deterministic.
However, as the numerical solution to a hard problem (described, for instance, by partial differential equations), $y$ does not have in general a closed-form expression.
Consequently, the data analyst looking for a behavioral representation, can legitimately think of $y$ as a "black-box" function: it is unknown for any particular configuration until the computer code is actually run.
It hence makes sense to postulate a prior model for $y$, expressing our initial belief regarding $y$ features.
Bayesian updating then combines the evidences acquired, the data output $ \left( y_i = y \left( x_i \right) \right)_{1 \leq i \leq n} $ , and the prior distribution to yield the posterior distribution.

Contrary to what is often implicitly stated in the literature, we believe that it is far from obvious to get a relevant prior representation for $y$.
Standard priors reflects more of their mathematical tractability than a real understanding of the phenomenon under study.
Incorporating doubtful prior information into the model may yield erroneous and overly confident forecasts, a dangerous cocktail in risk assessment.
For that reason, we focus in the following on a prior based on quite few assumptions about how the model output relates to its inputs; it is deliberately weakly informative.

%%%%%%%%%%%%%%%%%%%%%%%%%%%%%%%%%%%%%%%%%%%%%%%%%%%%%%%%%%%%%%%%%%%%%%%%%%%%%%%%%%%%%%%%%%%%%%%%%%%%%%%%%%%%%%%%%%%%%%%%%%%%%%%%%%%%%%%%%
% Gaussian field
%%%%%%%%%%%%%%%%%%%%%%%%%%%%%%%%%%%%%%%%%%%%%%%%%%%%%%%%%%%%%%%%%%%%%%%%%%%%%%%%%%%%%%%%%%%%%%%%%%%%%%%%%%%%%%%%%%%%%%%%%%%%%%%%%%%%%%%%%
\subsection{Gaussian field}
\label{GaussianField}
In this first sub-section, we point out that the conditioning of a Gaussian field generates a new Gaussian field and we recall classical formulae giving the conditional densities.

A Gaussian field $ \left( Y_x \right)_{x \in X} $ is characterized by the fact that all finite dimensional marginal distributions are Gaussian and by the following data:
\begin{itemize}
\item mean values: $ \mu \left( x \right) := \E \left( Y_x \right) $
\item covariances: $ \rho \left( x,x' \right) := \textrm{cov} \left( Y_x, Y_{x'} \right) $.
\end{itemize}

%%%%%%%%%%%%%%%%%%%%%%%%%%%%%%%%%%%%%%%%%%%%%%%%%%%%%%%%%%%%%%%%%%%%%%%%%%%%%%%%%%%%%%%%%%%%%%%%%%%%
% Theorem 1
%%%%%%%%%%%%%%%%%%%%%%%%%%%%%%%%%%%%%%%%%%%%%%%%%%%%%%%%%%%%%%%%%%%%%%%%%%%%%%%%%%%%%%%%%%%%%%%%%%%%
\begin{theorem}
Let $ x_1, x_2, \ldots, x_n \in X $ and $ y_1, y_2, \ldots, y_n \in \R $.\\
The conditional distribution of the Gaussian field $ \left( Y_x \right)_{x\in X} $ given $ \left( Y_{x_i} = y_i \right)_{1\leq i \leq n} $ is still Gaussian with
\begin{itemize}
\item mean values: 
\[ \E \left( Y_x\mid \left( Y_{x_i} = y_i \right)_{1 \leq i \leq n} \right) = \mu \left( x \right) + \b {\rho \left( x,\left( x_i \right) \right)} \b \Sigma^{-1} \left( \b y - \b {\mu \left( \left( x_i \right) \right)} \right)^T \]
\item covariances: 
\[ {\rm{cov}} \left( Y_x, Y_{x'} \mid \left( Y_{x_i} = y_i \right)_{1 \leq i \leq n} \right) = \rho \left( x,x' \right) - \b {\rho \left( x,\left( x_i \right) \right)} \b \Sigma^{-1} \b{ \rho \left( x',\left( x_i \right) \right)}^T \]
\end{itemize}
where
$ \b{ \rho \left( x,\left( x_i \right) \right)} := \left( \rho \left( x, x_1 \right), \rho \left( x, x_2 \right), \ldots, \rho \left( x, x_n \right) \right) $,
$ \b \Sigma := \left( \rho \left( x_i, x_j \right) \right)_{1 \leq i, j \leq n} $ is a positive-definite matrix,
$ \b y := \left( y_1, y_2, \ldots, y_n \right) $ and 
$ \b {\mu \left( \left( x_i \right) \right)} := \left( \mu \left( x_1 \right), \mu \left( x_2 \right), \ldots, \mu \left( x_n \right) \right) $.
\end{theorem}

This is a classical result in Probability Theory.

A standard solution to our initial problem is to use this theorem with a given a priori Gaussian field $ \left( Y_x \right)_{x \in X} $ where neither the mean value $ \mu = \mu \left( x \right) $ nor the variance $ \s^2 = \rho \left( x, x \right) $ depend on $x$.
In this case, we will denote $ k \left( x, x' \right) = \rho \left( x, x' \right) / \s^2 $, the correlation coefficient between $Y_x$ and $Y_{x'}$.
If the covariance depends only on the difference between $x$ and $x'$, the Gaussian field is said stationary or homogeneous (\cite{Abrahamsen}).
Moreover, if the covariance depends only on the Euclidean distance between $x$ and $x'$, the Gaussian field is said isotropic.

In practice, the mean $\mu$, the variance $\s^2$ and the correlation function $k$ are usually a priori defined by the expert or estimated by means of calibration methods.
The most commonly used is the maximum likelihood method which will be discussed in Section~\ref{MLE}.

Note that here and in the sequel, we avoid particular choices of the correlation $k$, vectors $ x_1, x_2, \ldots, x_n $ and numbers $ y_1, y_2, \ldots, y_n $ which could lead to degenerate distributions.

%%%%%%%%%%%%%%%%%%%%%%%%%%%%%%%%%%%%%%%%%%%%%%%%%%%%%%%%%%%%%%%%%%%%%%%%%%%%%%%%%%%%%%%%%%%%%%%%%%%%%%%%%%%%%%%%%%%%%%%%%%%%%%%%%%%%%%%%%
% Gaussian mixture with random mean
%%%%%%%%%%%%%%%%%%%%%%%%%%%%%%%%%%%%%%%%%%%%%%%%%%%%%%%%%%%%%%%%%%%%%%%%%%%%%%%%%%%%%%%%%%%%%%%%%%%%%%%%%%%%%%%%%%%%%%%%%%%%%%%%%%%%%%%%%
\subsection{Gaussian mixture with random mean}
\label{GaussianMixture}
In this sub-section, the prior model is a mixture of Gaussian processes, which differ by translation.
This way, it is not required to set a priori the mean value.
The posterior distribution of this parameter is deduced from the data via Bayesian inference.
We will see that a (non informative) improper prior for the mean parameter can be defined as the limit of a (weakly informative) proper uniform distribution.
The limiting step is totally justified here since, on the one hand, it leads to proper posteriors and, on the other hand, inferences have a positive miscalibration (see \cite{Gelman} for details), that is an overestimate (on average) of the variance.
The conditioned random field is Gaussian and we give formulae for the mean and variance.
Doing so, we recognize expressions initially established by \cite{Sacks} and also derived by \cite{Santner}.

We propose to consider a prior random field $ Y_x = U + W_x $ where $U$ is a real random variable following a uniform distribution on an interval $ \left[ -m, m \right] $ and where $ \left( W_x \right)_{x \in X} $ is a centered Gaussian field with constant variance.
Moreover, we suppose that $U$ and $ \left( W_x \right)_{x\in X} $ are independent.

The parameters characterizing the Gaussian field $ \left( W_x \right)_{x \in X} $ are the variance $ \s^2$ and the correlation function $k$ (recall that the mean is zero).
So, for all $ x, x' \in X $, $ \textrm{cov} \left( W_x, W_{x'} \right) = \s^2 k \left( x, x' \right) $ and, for all $ x\in X $, $ k \left (x, x\right) = 1 $.

Let $ x_1, x_2, \ldots, x_n \in X $ and $ \b y := \left( y_1, y_2, \ldots, y_n \right) \in \R^n $.
Denote $ \b \Sigma := \left( k \left( x_i,x_j \right) \right)_{1 \leq i,j \leq n} $ the positive-definite matrix of correlations and $ \b {k \left( x \right)} := \left( k \left( x,x_j \right) \right)_{1 \leq j \leq n} $ the correlation vector.

%%%%%%%%%%%%%%%%%%%%%%%%%%%%%%%%%%%%%%%%%%%%%%%%%%%%%%%%%%%%%%%%%%%%%%%%%%%%%%%%%%%%%%%%%%%%%%%%%%%%
% Theorem 2
%%%%%%%%%%%%%%%%%%%%%%%%%%%%%%%%%%%%%%%%%%%%%%%%%%%%%%%%%%%%%%%%%%%%%%%%%%%%%%%%%%%%%%%%%%%%%%%%%%%%
\begin{theorem}
The conditional distribution of the random field $ \left( Y_x \right)_{x\in X} $ knowing that $ \left( Y_{x_i} = y_i \right)_{1\leq i \leq n} $ is given by explicit formulae for the densities of finite dimensional marginals.\\
When the parameter $m$ goes to infinity, this conditional distribution becomes Gaussian.
In particular, when $m \rightarrow \infty$, the univariate conditional distribution of the random variable $Y_x$ becomes Gaussian with mean
\begin{equation}
\mu + \b {k \left( x \right)} \b \Sigma^{-1} \left( \b y - \mu \1 \right)^T \quad \textrm{with} \quad \mu := \frac{ \b y \b \Sigma^{-1} \1^T}{\1 \b \Sigma^{-1} \1^T}
\label{MeanGauss}
\end{equation}
and variance
\begin{equation}
\s^2 \left( 1 - \b {k \left( x \right)} \b \Sigma^{-1} \b {k \left( x \right)}^T + \frac{ \left( 1 - \1 \b \Sigma^{-1} \b {k \left( x \right)}^T \right)^2}{ \1 \b \Sigma^{-1} \1^T} \right)
\end{equation}
where $ \1 = \left( 1,1, \ldots,1 \right) \in \R^n $.
\end{theorem}

\begin{remark}
The mean $\mu + \b {k \left( x \right)} \b \Sigma^{-1} \left( \b y - \mu \1 \right)^T$ can also be written
\[ \left( \frac{ \1 \b \Sigma^{-1}}{\1 \b \Sigma^{-1} \1^T} \left( 1 - \b {k \left( x \right)} \b \Sigma^{-1} \1^T \right) + \b {k \left( x \right)} \b \Sigma^{-1} \right) \b y^T .\]
Note that an expression similar to (\ref{MeanGauss}) is obtained in Section~\ref{ComparisonStudentGauss}.
\end{remark}

%%%%%%%%%%%%%%%%%%%%%%%%%%%%%%%%%%%%%%%%%%%%%%%%%%%%%%%%%%%%%%%%%%%%%%%%%%%%%%%%%%%%%%%%%%%%%%%%%%%%
% Proof of Theorem 2
%%%%%%%%%%%%%%%%%%%%%%%%%%%%%%%%%%%%%%%%%%%%%%%%%%%%%%%%%%%%%%%%%%%%%%%%%%%%%%%%%%%%%%%%%%%%%%%%%%%%
\begin{proof}[Proof of Theorem 2]
We look for the distribution of the random field $ \left( Y_x \right)_{x\in X} $ given $ \left( Y_{x_i} = y_i \right)_{1 \leq i \leq n} $.\\
Let $r$ be a positive integer and $ \left( t_1,\ldots, t_r \right) \in X^r$.
We have:
\[ \left( Y_{t_1},\ldots,Y_{t_r},Y_{x_1},\ldots, Y_{x_n} \right) = \left( U,\ldots, U \right) + \left( W_{t_1},\ldots, W_{t_r}, W_{x_1},\ldots, W_{x_n} \right) \]
and $ \left( W_{t_1}, \ldots, W_{t_r}, W_{x_1}, \ldots, W_{x_n} \right)$ follows the distribution $ \mathcal{N} \left( 0, \b \Delta \right) $ with
\[ \b \Delta := \s^2 \begin{pmatrix} \b {\Sigma_2} & \b {k \left( t \right)} \\ \b {k \left( t \right)}^T & \b \Sigma \end{pmatrix} \textrm{, a positive-definite matrix} \]
where $ \b {k \left( t \right)} := \left( k \left( t_i,x_j \right) \right)_{1 \leq i \leq r,1 \leq j \leq n} $ and $ \b {\Sigma_2} := \left( k \left( t_i,t_j \right) \right)_{1 \leq i \leq r,1 \leq j \leq r} $.\\
Denote by $f$ the density of the random vector $ \left( Y_{t_1}, \ldots, Y_{t_r}, Y_{x_1}, \ldots, Y_{x_n} \right)$:
\[ f \left( \b \zeta \right) = \frac1{2m} \int_{-m}^m \frac1{\left(\sqrt{2\pi} \right)^{n+r} \sqrt{\left| \b \Delta\right|}} \exp \left( - \frac12 \left( \b \zeta - \b u \right) \b \Delta^{-1} \left( \b \zeta - \b u \right)^T \right) \textrm{ d}u \]
where $ \b \zeta := \left( y_{t_1}, y_{t_2}, \ldots, y_{t_r}, y_{x_1}, y_{x_2}, \ldots, y_{x_n} \right) $ and $ \b u := \left( u,u, \ldots, u \right) \in \R^{n+r} $.

The conditional density of $ \left( Y_{t_1}, Y_{t_2}, \ldots, Y_{t_r} \right) $ given $ \left( Y_{x_i} = y_{x_i} \right)_{1 \leq i \leq n} $ is
\[ \b z := \left( y_{t_1}, y_{t_2}, \ldots, y_{t_r} \right) \mapsto \frac{f \left( \b z, \b y \right)}{\int_{\R^r} f \left( \b {z'}, \b y \right) \textrm{ d} \b {z'}} =: g \left( \b z \right) .\]
After simplification, we get:
\[ g \left( \b z \right) = \frac{ \int_{-m}^m \exp \left( - \frac12 \left( \begin{pmatrix} \b z & \b y \end{pmatrix} - \b u \right) \b \Delta^{-1} \left( \begin{pmatrix} \b z & \b y \end{pmatrix} - \b u \right)^T \right) \textrm{ d}u }{\int_{\R^r} \int_{-m}^m \exp \left( - \frac12 \left( \begin{pmatrix} \b {z'} & \b y \end{pmatrix} - \b u \right) \b \Delta^{-1} \left( \begin{pmatrix} \b {z'} & \b y \end{pmatrix} - \b u \right)^T \right) \textrm{ d} \b {z'} \textrm{ d}u } .\]
By monotone convergence, some calculations show that:
\[ g \left( \b z \right) \underset{m\to+\infty}{\longrightarrow} \frac{\exp \left( - \frac12 \left( \begin{pmatrix} \b z & \b y \end{pmatrix} \b \Delta^{-1} \begin{pmatrix} \b z & \b y \end{pmatrix}^T - \frac{\left(\1 \b \Delta^{-1} \begin{pmatrix} \b z & \b y \end{pmatrix}^T \right)^2}{ \1 \b \Delta^{-1} \1^T} \right) \right) }{\int_{\R^r} \exp \left( - \frac12 \left( \begin{pmatrix} \b {z'} & \b y \end{pmatrix} \b \Delta^{-1} \begin{pmatrix} \b {z'} & \b y \end{pmatrix}^T - \frac{ \left( \1 \b \Delta^{-1} \begin{pmatrix} \b {z'} & \b y \end{pmatrix}^T \right)^2}{ \1 \b \Delta^{-1} \1^T} \right) \right) \textrm{ d} \b {z'} } .\]
Within the exponential in the numerator of this expression we identify a non-negative second degree polynomial of the variable $ \b z$.
We recognize a Gaussian distribution (of dimension $r$).\\
In the limit situation $m\rightarrow\infty$, we note that the distribution of the conditioned random field is well-defined and is Gaussian.
Some calculations show that, at the point $x \in X$, the one-dimensional marginal Gaussian distribution of the field has mean
\[ \mu + \b {k \left( x \right)} \b \Sigma^{-1} \left( \b y - \mu \1 \right)^T  \quad \textrm{with} \quad \mu := \frac{ \b y \b \Sigma^{-1} \1^T}{\1 \b \Sigma^{-1} \1^T} \]
and variance
\[ \s^2 \left( 1 - \b {k \left( x \right)} \b \Sigma^{-1} \b {k \left( x \right)}^T + \frac{ \left( 1 - \1 \b \Sigma^{-1} \b {k \left( x \right)}^T \right)^2}{\1 \b \Sigma^{-1} \1^T} \right) .\]
\end{proof}

%%%%%%%%%%%%%%%%%%%%%%%%%%%%%%%%%%%%%%%%%%%%%%%%%%%%%%%%%%%%%%%%%%%%%%%%%%%%%%%%%%%%%%%%%%%%%%%%%%%%%%%%%%%%%%%%%%%%%%%%%%%%%%%%%%%%%%%%%
% Gaussian mixture with random mean and variance
%%%%%%%%%%%%%%%%%%%%%%%%%%%%%%%%%%%%%%%%%%%%%%%%%%%%%%%%%%%%%%%%%%%%%%%%%%%%%%%%%%%%%%%%%%%%%%%%%%%%%%%%%%%%%%%%%%%%%%%%%%%%%%%%%%%%%%%%%
\subsection{Gaussian mixture with random mean and variance}
\label{StudentField}
We go one step further, avoiding the choice of the mean and variance.
The prior model is a mixture of Gaussian processes, which differ by affine transformation.
Once again, the conditioned random field is well-defined selecting (weakly informative) proper uniform priors for the unknown mean and standard deviation.
When the distributions supports respectively tend to the entire and positive real lines, the multivariate Student distribution arises as the random field posterior distribution.
We give explicit formulae for the location and scale parameters.
\cite{Santner} and \cite{Kato} describe similar models with different priors.
For example, with Jeffreys priors for the mean and standard deviation (scale-invariance property), Santner and al. come to the same posterior except for the number of degrees of freedom.
We justify our choice arguing that a positive miscalibration is always preferred by the risk-averse individual.

%%%%%%%%%%%%%%%%%%%%%%%%%%%%%%%%%%%%%%%%%%%%%%%%%%%%%%%%%%%%%%%%%%%%%%%%%%%%%%%%%%%%%%%%%%%%%%%%%%%%%%%%%%%%%%%%%%%%%%%%%%%%%%%%%%%%%%%%%
% Multivariate Student distribution
%%%%%%%%%%%%%%%%%%%%%%%%%%%%%%%%%%%%%%%%%%%%%%%%%%%%%%%%%%%%%%%%%%%%%%%%%%%%%%%%%%%%%%%%%%%%%%%%%%%%%%%%%%%%%%%%%%%%%%%%%%%%%%%%%%%%%%%%%
\subsubsection{Multivariate Student distribution}
First, recall the definition of a multivariate Student distribution.
We refer to the book of \cite{Kotz}.
A $p$-variate Student distribution (or a multivariate t-distribution) has density:
\[ \b t \mapsto \frac{1}{ \left(\sqrt{\pi\nu}\right)^p \sqrt{\left| \b \Sigma \right|}} \frac{\Gamma \left( \frac{\nu+p}2 \right) }{\Gamma \left( \frac{\nu}2 \right)} \left( 1 + \frac1{\nu} \left( \b t - \b \mu \right) \b \Sigma^{-1} \left( \b t - \b \mu \right)^T \right)^{-\frac{\nu+p}2} \quad \left( \b t \in \R^p \right) \]
where the positive integer $\nu$ is the number of degrees of freedom, $ \b \Sigma$ is the $ p \times p $ positive-definite matrix of scale parameters and $ \b \mu$ is the $ 1 \times p $ vector of location parameters.

For $ \nu > 1 $, the mean vector of the Student distribution is well-defined and equals to $ \b \mu$.
For $ \nu > 2 $, the covariance matrix of the Student distribution is well-defined and equals to $ \frac{\nu}{\nu-2} \b \Sigma $.

This is a multi-dimensional generalization of the Student distribution.
When $ \nu = 1 $, the distribution is a multivariate Cauchy distribution.
When $ \nu $ goes to infinity, the distribution tends to a multivariate Gaussian distribution.

%%%%%%%%%%%%%%%%%%%%%%%%%%%%%%%%%%%%%%%%%%%%%%%%%%%%%%%%%%%%%%%%%%%%%%%%%%%%%%%%%%%%%%%%%%%%%%%%%%%%%%%%%%%%%%%%%%%%%%%%%%%%%%%%%%%%%%%%%
% Student field posterior
%%%%%%%%%%%%%%%%%%%%%%%%%%%%%%%%%%%%%%%%%%%%%%%%%%%%%%%%%%%%%%%%%%%%%%%%%%%%%%%%%%%%%%%%%%%%%%%%%%%%%%%%%%%%%%%%%%%%%%%%%%%%%%%%%%%%%%%%%
\subsubsection{Student field posterior}
\label{StudentFieldConstruction}
Now, we propose to consider a prior random field $ Y_x = U + V W_x $, where $U$ is a real random variable following a uniform distribution on an interval $ \left[ -m,m \right] $, where $V$ is a real and positive random variable following a uniform distribution on an interval $ \left[ \e, 1/\e \right] $ and where $ \left( W_x \right)_{x \in X} $ is a centered normalized Gaussian field.
Moreover, we suppose that $U$, $V$ and $ \left( W_x \right)_{x\in X} $ are independent.

The parameter characterizing the Gaussian field $ \left( W_x \right)_{x \in X} $ is the correlation function $k$ (recall that the mean is zero and the variance is 1).
We suppose here that $n \geq 3$.
Let $ x_1, x_2, \ldots, x_n \in X $ and $ \b y := \left(y_1,y_2,\ldots,y_n\right) \in \R^n$.
Denote $ \b \Sigma := \left( k \left( x_i,x_j \right) \right)_{1 \leq i,j \leq n} $ the positive-definite matrix of correlations and $ \b {k \left( x \right)} := \left( k \left( x,x_j \right) \right)_{1 \leq j \leq n} $ the correlation vector.

%%%%%%%%%%%%%%%%%%%%%%%%%%%%%%%%%%%%%%%%%%%%%%%%%%%%%%%%%%%%%%%%%%%%%%%%%%%%%%%%%%%%%%%%%%%%%%%%%%%%
% Theorem 3
%%%%%%%%%%%%%%%%%%%%%%%%%%%%%%%%%%%%%%%%%%%%%%%%%%%%%%%%%%%%%%%%%%%%%%%%%%%%%%%%%%%%%%%%%%%%%%%%%%%%
\begin{theorem}
The conditional distribution of the random field $ \left( Y_x \right)_{x\in X} $ knowing that $ \left( Y_{x_i} = y_i \right)_{1\leq i \leq n} $ is given by explicit formulae of densities of finite dimensional marginals.\\
When the parameter $m$ goes to infinity and $\e$ goes to zero, for $ n > 2 $, this conditional distribution becomes a multivariate Student distribution.\\
In particular, when $ m \rightarrow \infty $, $ \e \rightarrow 0 $ and $ n > 2 $, the univariate conditional distribution of the random variable $Y_x$ becomes a Student distribution with $n-2$ degrees of freedom, with location parameter
\[ \mu + \b {k \left( x \right)} \b \Sigma^{-1} \left( \b y - \mu \1 \right)^T \quad \textrm{with} \quad \mu := \frac{ \b y \b \Sigma^{-1} \1^T}{\1 \b \Sigma^{-1} \1^T} \]
and scale parameter
\[ \sqrt{\frac1{n-2} \left( \left( \b y - \mu \1 \right) \b \Sigma^{-1} \b y^T \right) \left( 1 - \b {k \left( x \right)} \b \Sigma^{-1} \b {k \left( x \right)}^T + \frac{ \left( 1 - \1 \b \Sigma^{-1} \b {k \left( x \right)}^T \right)^2}{\1 \b \Sigma^{-1} \1^T} \right) } \]
where $ \1 = \left( 1,1, \ldots,1 \right) \in \R^n $.
\end{theorem}

The proof of Theorem 3 results from the same reasoning than the proof of Theorem 2 with more complex formal calculations.
It is developed in the supplementary materials available on line.

\begin{remark}
The expression $ \left( \b y - \mu \1 \right) \b \Sigma^{-1} \b y^T $ can also be written $ \left( \b y - \mu \1 \right) \b \Sigma^{-1} \left( \b y - \mu \1 \right)^T $.
Note that a similar expression is obtained in Section~\ref{ComparisonStudentGauss}.
\end{remark}

%%%%%%%%%%%%%%%%%%%%%%%%%%%%%%%%%%%%%%%%%%%%%%%%%%%%%%%%%%%%%%%%%%%%%%%%%%%%%%%%%%%%%%%%%%%%%%%%%%%%%%%%%%%%%%%%%%%%%%%%%%%%%%%%%%%%%%%%%
% Model implementation
%%%%%%%%%%%%%%%%%%%%%%%%%%%%%%%%%%%%%%%%%%%%%%%%%%%%%%%%%%%%%%%%%%%%%%%%%%%%%%%%%%%%%%%%%%%%%%%%%%%%%%%%%%%%%%%%%%%%%%%%%%%%%%%%%%%%%%%%%
\subsection{Model implementation}
\label{ModelImplementation}

%%%%%%%%%%%%%%%%%%%%%%%%%%%%%%%%%%%%%%%%%%%%%%%%%%%%%%%%%%%%%%%%%%%%%%%%%%%%%%%%%%%%%%%%%%%%%%%%%%%%%%%%%%%%%%%%%%%%%%%%%%%%%%%%%%%%%%%%%
% Discussion on correlation functions
%%%%%%%%%%%%%%%%%%%%%%%%%%%%%%%%%%%%%%%%%%%%%%%%%%%%%%%%%%%%%%%%%%%%%%%%%%%%%%%%%%%%%%%%%%%%%%%%%%%%%%%%%%%%%%%%%%%%%%%%%%%%%%%%%%%%%%%%%
\subsubsection{Discussion on correlation functions}
The choice a priori of the correlation function is one of the major difficulties in random field modeling (\cite{Sacks}) and the discussion is unfortunately often avoided.
It is sometimes justified by expert knowledge (\cite{CornFord}) in particular applied cases.
\cite{Rasmussen} give a complete list of correlation functions, classified according to practical considerations.
The most popular correlation families are isotropic: Exponential, Matern and Rational Quadratic classes (\cite{Abrahamsen}).
Indeed, an isotropic random field does not suffer the curse of dimensionality since the number of parameters in the model does not depend on the factor space dimension ($D$).

Determining appropriate values for the parameters of the correlation function is the purpose of calibration methods.
The most popular of them is the maximum likelihood estimation (MLE) described in the next sub-section.
\cite{Rasmussen}, \cite{Stein} and \cite{Robert} describe this method and discuss about its capabilities and limitations such as numerical issues for too big samples, multiple optima for too small samples and over-fitting problems for which the sample is well learned but the unknown function values are poorly predicted everywhere else in the factor space.

In order to overtake these different limitations, a good way seems to try different choices of correlation function; each alternative is then evaluated by checking if known values of the responses are compatible with the confidence intervals predicted by the random field.
The literature describes several types of cross-validation methods measuring the predictive capability of a model (\cite{Currin}).
It is of course necessary to use a set of data which is not involved in the construction of the a posteriori model.
Because we cannot, when information is scarce, afford not to include all the gathered data in the final model, the partition into training and validation sets is only temporary, for the purpose of the cross-validation stage.
In that case, the test is therefore performed on an "incomplete" posterior model version.

\cite{Rasmussen} describe an alternative method which consists in choosing a parameterized family of correlation function and a prior distribution on its parameters, to construct a hierarchical model.
This method involves analytical approximations of integrals.
Markov chain Monte Carlo (MCMC) methods are popular solutions to make these computations, see \cite{Robert} for detailed description of these methods.
A major drawback is here the cost of such calculations.

In fact, there is no complete and rigorous study for the identification of the correlation function giving the best results in a particular situation.
However, \cite{Abrahamsen} lists some geometrical properties (continuity and differentiability) of the correlation functions which can guide the expert.

%%%%%%%%%%%%%%%%%%%%%%%%%%%%%%%%%%%%%%%%%%%%%%%%%%%%%%%%%%%%%%%%%%%%%%%%%%%%%%%%%%%%%%%%%%%%%%%%%%%%%%%%%%%%%%%%%%%%%%%%%%%%%%%%%%%%%%%%%
% Maximum Likelihood Estimation
%%%%%%%%%%%%%%%%%%%%%%%%%%%%%%%%%%%%%%%%%%%%%%%%%%%%%%%%%%%%%%%%%%%%%%%%%%%%%%%%%%%%%%%%%%%%%%%%%%%%%%%%%%%%%%%%%%%%%%%%%%%%%%%%%%%%%%%%%
\subsubsection{Maximum Likelihood Estimation}
\label{MLE}
The classical maximum likelihood estimation (MLE) method defines an estimator of the unknown parameter vector $ \b \theta $ of a probability distribution $ f_{\b \theta} $.
This estimator is the value $ \b {\theta_{\max}} $ which maximizes the density distribution of a random sample calculated at the observed value of this sample.

Under the same name, this method has been adapted to the identification of the parameters of an unknown random field $ \left( Y_x \right)_{x\in X} $ when considering a family of values $ Y_{x_1}, Y_{x_2}, \ldots, Y_{x_n} $.
It is important to have in mind that the observation set is a single outcome of the random field so that the validity of the MLE method, in this context, is not obvious.
However, this topic is outside the scope of this article.

Considering now the probabilistic modeling of a deterministic phenomenon, the MLE provides a practical procedure to set the parameters of the a priori random field $ \left( Y_x \right)_{x\in X} $ knowing $ Y_{x_i} = y_i $, $ 1 \leq i \leq n $.
This is a classical approach in Gaussian field modeling, see \cite{Rasmussen}, \cite{Stein} and \cite{Robert}.

Even if our justifications are incomplete, let us describe the MLE for the random model described in Section~\ref{StudentField}.

First, set a parametric family of correlation functions $ k_{ \b \theta} $ depending on the parameter vector $\b \theta$, which defines also the parametric correlation matrix $ \b {\Sigma_\theta } = \left( k_{ \b \theta} \left( x_i, x_j \right) \right)_{1\leq i,j\leq n} $.
The MLE consists in choosing $\b \theta$ which maximizes the density of the random vector $ \left( Y_{x_1}, Y_{x_2}, \ldots, Y_{x_n} \right) $ according to $m$ and $\e$:
\[ f_{m, \e} \left( \b \theta \right) = \frac1{2m} \frac1{\frac1 \e - \e} \int_{-m}^m \int_\e^{\frac1 \e} \frac1{\left(\sqrt{2\pi} \right)^n v^n \sqrt{\left| \b {\Sigma_\theta }\right|}} \exp \left( - \frac1{2v^2} \left( \b y - \b u \right) \b {\Sigma_\theta }^{-1} \left( \b y - \b u \right)^T \right) \textrm{ d}u \textrm{ d}v \]
with $ \b y := \left( y_{x_1}, y_{x_2}, \ldots, y_{x_n} \right) $ and $ \b u := \left( u,u, \ldots, u \right) \in \R^n $.

The maximum likelihood estimator $ \b { \theta_{\max} \left( m, \e \right)} $ maximizes $ f_{m, \e} \left( \b \theta \right) $.
It maximizes also:
\[ \tilde f_{m, \e} \left( \b \theta \right) = \int_{-m}^m \int_\e^{\frac1 \e} \frac1{v^n \sqrt{\left| \b {\Sigma_\theta }\right|}} \exp \left( - \frac1{2v^2} \left( \b y - \b u \right) \b {\Sigma_\theta }^{-1} \left( \b y - \b u \right)^T \right) \textrm{ d}u \textrm{ d}v .\]
We don't know any analytic expression of $ \tilde f_{m, \e} \left( \b \theta \right) $ and a fortiori of $ \b { \theta_{\max} \left( m, \e \right)} $ but since we are interested in large values of $m$ and little values of $\e$, we consider the limit value $\tilde f_{\infty, 0} \left( \b \theta \right)$.
So we propose to study:
\[ \tilde f_{\infty, 0} \left( \b \theta \right) = \int_{-\infty}^{+\infty} \int_0^{+\infty} \frac1{v^n \sqrt{ \left| \b {\Sigma_\theta} \right|}} \exp \left( - \frac1{2v^2} \left( \b y- \b u \right) \b {\Sigma_\theta }^{-1} \left( \b y - \b u \right)^T \right) \textrm{ d}u \textrm{ d}v .\]
A short calculation gives 
\[ \tilde f_{\infty,0} \left( \b \theta \right) = \frac{ \sqrt{\pi} 2^{\frac{n-3}2} \left(n \s^2_{ \b \theta}\right)^{\frac{2-n}2} \Gamma \left( \frac{n-2}2 \right)}{\sqrt{ \left| \b {\Sigma_\theta} \right|} \left( \1 \b {\Sigma_\theta }^{-1} \1^T \right)} \]
with
\[ \s^2_{ \b \theta} := \frac1n \left( \b y - \mu_{ \b \theta} \1 \right) \b {\Sigma_\theta }^{-1} \left( \b y - \mu_{ \b \theta} \1 \right)^T \quad \textrm{and} \quad \mu_{ \b \theta} := \frac{\1 \b {\Sigma_\theta }^{-1} \b y^T}{\1 \b {\Sigma_\theta }^{-1} \1^T} .\]
Maximizing $ \tilde f_{\infty,0} \left( \b \theta \right) $ is equivalent to minimize $ - \ln \left( \tilde f_{\infty,0} \left( \b \theta \right) \right)$ and after simplifications, we get:
\[ \b { \theta_{\max} \left( \infty, 0 \right)} \quad \textrm{minimizes} \quad n \ln \left( \s^2_{ \b \theta} \right) + \ln \left( \left| \b {\Sigma_\theta} \right| \right) + 2 \ln \left( \frac {\1 \b {\Sigma_\theta }^{-1} \1^T}{ \s^2_{ \b \theta}} \right) .\]
Observe the difference with the Gaussian case in which the MLE method proposes analytical expressions for the mean and the variance of the field and where the parameter vector $\b \theta$ of the correlation function is estimated by minimizing:
\[ n \ln \left( \s^2_{ \b \theta} \right) + \ln \left( \left| \b {\Sigma_\theta} \right| \right) .\]

%%%%%%%%%%%%%%%%%%%%%%%%%%%%%%%%%%%%%%%%%%%%%%%%%%%%%%%%%%%%%%%%%%%%%%%%%%%%%%%%%%%%%%%%%%%%%%%%%%%%%%%%%%%%%%%%%%%%%%%%%%%%%%%%%%%%%%%%%
% Examples and results
%%%%%%%%%%%%%%%%%%%%%%%%%%%%%%%%%%%%%%%%%%%%%%%%%%%%%%%%%%%%%%%%%%%%%%%%%%%%%%%%%%%%%%%%%%%%%%%%%%%%%%%%%%%%%%%%%%%%%%%%%%%%%%%%%%%%%%%%%
\subsection{Examples and results}

%%%%%%%%%%%%%%%%%%%%%%%%%%%%%%%%%%%%%%%%%%%%%%%%%%%%%%%%%%%%%%%%%%%%%%%%%%%%%%%%%%%%%%%%%%%%%%%%%%%%%%
% Comparison between Student and Gaussian fields
%%%%%%%%%%%%%%%%%%%%%%%%%%%%%%%%%%%%%%%%%%%%%%%%%%%%%%%%%%%%%%%%%%%%%%%%%%%%%%%%%%%%%%%%%%%%%%%%%%%%%%
\subsubsection{Comparison between Student and Gaussian fields}
\label{ComparisonStudentGauss}
Let us come back for a while to the Gaussian field model described in Section~\ref{GaussianField}, choosing constant mean and variance.
If the mean $\mu$ and variance $\s^2$ are determined by the MLE method considering the data set $ \left( y \left( x_i \right) = y_i \right)_{1 \leq i \leq n} $, then the following formulae are obtained:
\[ \mu = \frac{ \b y \b \Sigma^{-1} \1^T}{\1 \b \Sigma^{-1} \1^T} \quad \textrm{and} \quad \s^2 = \frac1{n} \left( \b y - \mu \1 \right) \b \Sigma^{-1} \left( \b y - \mu \1 \right)^T .\]
This is a classical result in this area.
See for example \cite{Currin} or \cite{Stein}.

Following Theorem~1, the a posteriori mean and variance of the Gaussian random variable $Y_x$ are respectively
\[ \mu + \b {k \left( x \right)} \b \Sigma^{-1} \left( \b y - \mu \1 \right)^T \quad \textrm{and} \quad \s^2 \left( 1 - \b {k \left( x \right)} \b \Sigma^{-1} \b {k \left( x \right)}^T \right) .\]
It is interesting to compare these values with those given in Theorem 3.
Note that, for the same correlation function, the mean remains the same and the variance is greater.
We can illustrate this difference on a simple example in one dimension.

Let $ X = \left[ -5,5 \right] $.
Let a sample of 5 points $ \left( x_1 = -4, x_2 = -3, x_3 = -1, x_4 = 0, x_5 = 2 \right) $ with the associated response values $ \left( y_1 = -2, y_2 = 0, y_3 = 1, y_4 = 2, y_5 = -1 \right) $.
We build for an arbitrary fixed correlation function, $ \rho \left( x,x' \right) := \exp \left( - 100 \left| x - x' \right|^2 \right) $, a Gaussian field and a Student field given the knowledge of the previous sample.

\begin{figure}[!h]
\begin{center}
\includegraphics[width=8.9cm,height=6cm,angle=0]{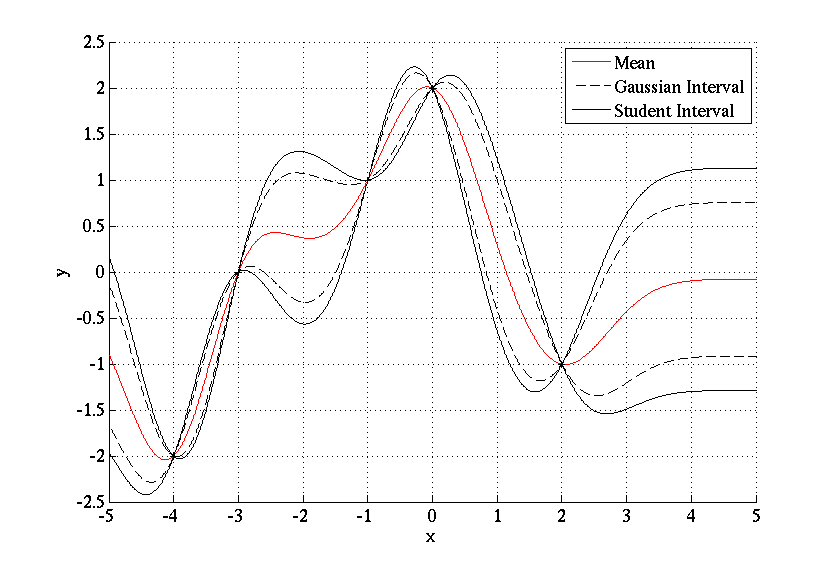}
\end{center}
\caption{Univariate illustration. Gaussian and Student field posteriors derived from five evaluations of the unknown function $ y \left( x \right) $.}
\label{figureGaussianStudent}
\end{figure}

Figure~\ref{figureGaussianStudent} shows, for each field after conditioning, the mean and the boundaries of an interval computed with a confidence of 0.90.
As expected, the confidence intervals are both null at the 5 sampled points since the distributions at these points degenerate in Dirac distributions.
The confidence interval of the Student posterior is larger than the Gaussian posterior equivalent, suggesting a lower degree of belief in the prediction.
The Student field seems more reliable, since we take into account all the possible values of the mean $\mu$ and the variance $ \s^2$, that were arbitrarily fixed in the case of the Gaussian field.

%%%%%%%%%%%%%%%%%%%%%%%%%%%%%%%%%%%%%%%%%%%%%%%%%%%%%%%%%%%%%%%%%%%%%%%%%%%%%%%%%%%%%%%%%%%%%%%%%%%%%%
% Prediction error study
%%%%%%%%%%%%%%%%%%%%%%%%%%%%%%%%%%%%%%%%%%%%%%%%%%%%%%%%%%%%%%%%%%%%%%%%%%%%%%%%%%%%%%%%%%%%%%%%%%%%%%
\subsubsection{Prediction error study}
\label{ErrorStudy}
In order to highlight the predictive capabilities of the Student field, we use an analytical benchmark function~$y$, the 5-dimensional quadric example described in Appendix~\ref{quadric_example}, to study the evolution of the model prediction error as the number of observed points increases.

The Student posterior field $ \left( Y_x \right)_{x\in X} $ is derived from a set of random samples $ \left( x_i, y_i = y \left( x_i \right) \right)_{1 \leq i \leq n} $ according to the results of the Section~\ref{StudentFieldConstruction}.
We selected an anisotropic $\g$-exponential correlation function (refer to Appendix~\ref{correlationfunction}) whose free parameters were set applying the MLE method, presented in Section~\ref{MLE}.

The following quantities are considered as quality criteria.
For any $ t \in X $, we define the relative error to be 
\[ \epsilon_r \left( t \right) := \left| \frac{\E \left( Y_t \right) - y \left( t \right)}{y \left( t \right)} \right| .\]
We call mean relative error, the integral over $X$ of $\epsilon_r$ and we call ccdf, the complementary cumulative density function of $\epsilon_r$, that is
\[ \alpha \mapsto P \left( \left\{ t \in X \mid \epsilon_r \left( t \right) > \alpha \right\} \right) .\]
These statistics have been calculated for an increasing number of data, $n$, and the results are plotted on Figure~\ref{ErrorQuadric}.
As expected, the relative error falls.
Considering a sample of 10 data, $\epsilon_r$ is less than $ 42\% $ with a probability of 0.9 (poor model).
This value drops to $ 1.9\% $ with 100 data and to $ 0.21\% $ with 1000 data.

\begin{figure}[!h]
\includegraphics[width=3.2in,height=2.49in,angle=0]{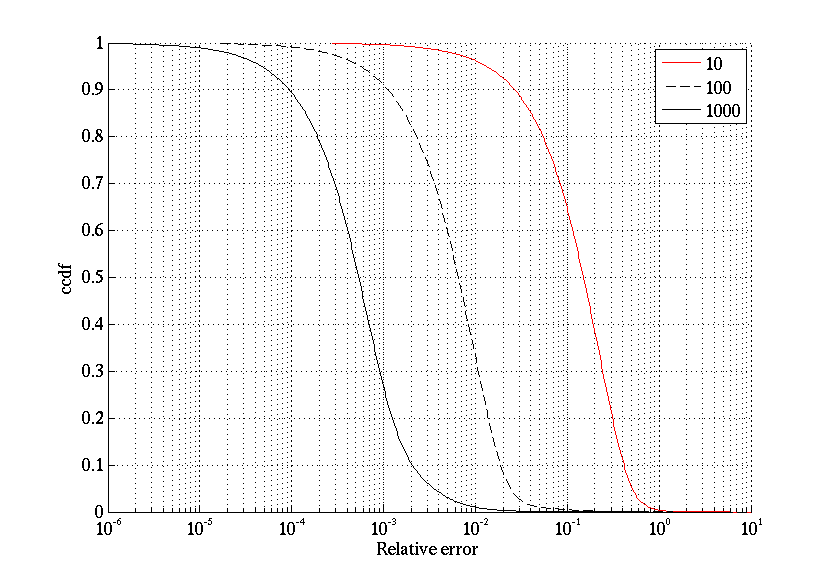} \hfill
\includegraphics[width=3.2in,height=2.49in,angle=0]{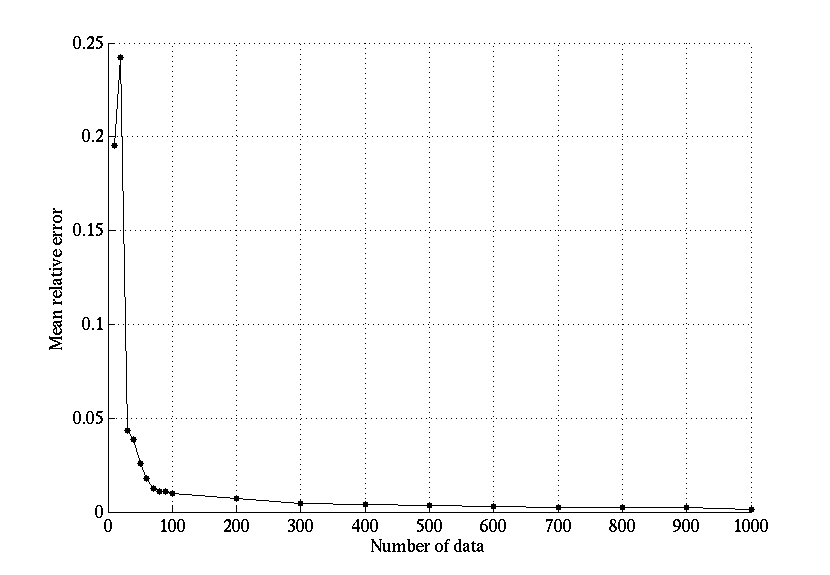}
\caption{Convergence rates for a Student field approximating the quadric benchmark function.}
\label{ErrorQuadric}
\end{figure}

%%%%%%%%%%%%%%%%%%%%%%%%%%%%%%%%%%%%%%%%%%%%%%%%%%%%%%%%%%%%%%%%%%%%%%%%%%%%%%%%%%%%%%%%%%%%%%%%%%%%%%%%%%%%%%%%%%%%%%%%%%%%%%%%%%%%%%%%%
% Random fields application
%%%%%%%%%%%%%%%%%%%%%%%%%%%%%%%%%%%%%%%%%%%%%%%%%%%%%%%%%%%%%%%%%%%%%%%%%%%%%%%%%%%%%%%%%%%%%%%%%%%%%%%%%%%%%%%%%%%%%%%%%%%%%%%%%%%%%%%%%
\section{Random fields application}
\label{Applications}

%%%%%%%%%%%%%%%%%%%%%%%%%%%%%%%%%%%%%%%%%%%%%%%%%%%%%%%%%%%%%%%%%%%%%%%%%%%%%%%%%%%%%%%%%%%%%%%%%%%%%%%%%%%%%%%%%%%%%%%%%%%%%%%%%%%%%%%%%
% Failure risk probability
%%%%%%%%%%%%%%%%%%%%%%%%%%%%%%%%%%%%%%%%%%%%%%%%%%%%%%%%%%%%%%%%%%%%%%%%%%%%%%%%%%%%%%%%%%%%%%%%%%%%%%%%%%%%%%%%%%%%%%%%%%%%%%%%%%%%%%%%%
\subsection{Failure risk probability}
\label{FailureRisk}
The \textit{real failure risk} is the measure of the \textit{real failure set} $ \left\{ x \in X \mid y \left( x \right) \in A \right\} $
where $ X \subseteq \R^D $ is the factor space provided with a probability distribution $P$.
In a lot of practical cases, the function $ x \mapsto y \left( x \right) $ is not available though; it can only be sampled at a few $x$ locations.
According to Section~\ref{RandomModels}, a random field $ \left( Y_x \right)_{x \in X} $ defined on a probability space $ \left( \Omega, \P \right) $ is derived from a set of observations in order to approximate the function $y$.

This modeling stage is not without consequences on the approach retained to calculate the failure risk.
Since the knowledge on the function $y$ is low, it seems reasonable to provide the failure risk prediction with a confidence measure.
Thus a decision-maker could easily evaluate the model quality and rule on the manufactured product robustness more safely.
This problem can be formulated as follows: how can we propagate the uncertainty inherent to any prediction based on a random field to the failure risk estimation?

If the phenomenon under study was a real random field $ \left( Y_x \right)_{x\in X} $, we could consider the random variable
\[ \omega \mapsto P \left( \left\{ x \in X \mid Y_x \left( \omega \right) \in A \right\} \right) \quad \textrm{on } \left( \Omega, \P \right) \]
as the \textit{failure risk} and its distribution as the \textit{failure risk probability}.

Let us explain why it is not the good point of view in the case where the phenomenon under study is deterministic.
Strictly speaking, the model $ \left( Y_x \right)_{x \in X} $ cannot describe the reality, which is perfectly determined (but unknown).
It correctly represents our knowledge at the observation points.
Besides, we have to give an interpretation to the process randomness.
If we retain the subjective conception of probability theory, randomness results from incomplete knowledge of the quantity $ y \left( x \right) $.
Within this framework, the measure~$\P$ quantifies a degree of belief regarding the model forecasts.
Randomness becomes a way to assess the quality, i.e. the predictive capability of the model.
The expression $ \P \left( Y_x \in A \right) = 0.5 $ does not mean that the point $x$ has a one-in-two chance of belonging to the real failure set; it means that, regarding the available information, one decision-maker out of two will consider (believe) that $x$ belong to the real failure set.
The model is therefore inadequate at point $x$.

If the previous definition of the failure risk is absolutely correct mathematically, it cannot be used, in this particular case, to model incomplete knowledge.
An example is necessary here to clarify our point.
We claim that the condition $\P \left(Y_x\in A\right)=1/2$ for any $ x\in X $, which is the sign of a poor model, is not incompatible with the fact that $P\left( Y_.\left(\omega\right)\in A\right) = 1/2$ for any $\omega \in \Omega$.
However, a Dirac mass distribution, or a distribution close to a Dirac mass, represents a phenomenon for which the information is sure, or close to sure.
This does not reflect the degree of insecurity associated to an inadequate model.
Consequently, we conclude that the distribution of the random variable $P\left( Y_.\left(\omega\right)\in A\right)$ does not provide relevant information regarding the model quality.

It is not difficult to prove our claim: remark that if the random variables $Y_x$ are pairwise independent and if the measure $P$ is continuous, then the distribution of $P\left( Y_.\left(\omega\right)\in A\right)$ is a Dirac mass.
Thus, just consider a model fulfilling this independence condition and such that $\P \left(Y_x\in A\right)=1/2$ for any $ x\in X $.

The random field alone is not a knowledge representation; it cannot encode the mental procedures involved in belief assessments.
An additional analysis is hence necessary.
We certainly do not have at our disposal a strict condition of membership to the real failure set.
Nevertheless, we can easily compute the following membership function $\Mi$:
\[ \left\{ \begin{array}{l}
\left( X, P \right) \to \left[ 0,1\right] \\
x \mapsto \Mi \left( x \right) := \P \left( Y_x \in A \right).
\end{array} \right. \]
Such a function associates to each point $ x \in X $ a real number in the interval $ \left[ 0, 1 \right] $ measuring the grade of membership of $x$ to the failure set.
Thus, the crisp failure set, unknown, can be only imprecisely characterized by the membership function $\Mi$.
Due to incomplete knowledge, we can only access to the fuzzy version of the failure set.
According to \cite{Zadeh}, "such a framework provides a natural way of dealing with problems in which the sources of imprecision is the absence of sharply defined criteria of class membership rather than the presence of random variables."

By definition, the notion of "belonging" is not well-defined for a fuzzy set.
For a given grade of membership $ \Mi \left( x \right) $, a decision-maker may consider that $x$ belongs to the failure set whereas another may not, depending on its own risk tolerance.
In order to account for the subjective nature of the decision, we introduce the level $\a$, $ 0 \leq \a \leq 1 $, and agree to say that "$x$ belong the failure set" if $ \Mi \left( x \right) > \a $.
The threshold $\a$ determines the status to give to uncertain predictions.
These data are therefore aggregated in two classes: "good" and "poor".
Let us define the $\a$-level failure set:
\[ \left\{ x \in X \mid \Mi \left( x \right) > \a \right\} .\]
It is a crisp subset of $X$ and can be consequently measured:
\[ R \left( \a \right) := P \left( \left\{ x \in X \mid \Mi \left( x \right) > \a \right\} \right) .\]
It is the risk probability estimation by a decision-maker whose risk tolerance is $\a$.
Its distribution is a Dirac mass at the point $ R \left( \a \right) $, denoted by $ \d_{R \left( \a \right)} $.
We implicitly define here a causal model: the risk probability is conditional on the risk tolerance $\a$ (\cite{Pearl}).
Being unsure of the decision-maker risk tolerance, we consider $\a$ as a random variable on $ \left( \left[ 0,1 \right], \eta\right) $.
The distribution $\eta$ stands for the risk tolerance distribution of the decision-maker.
In this framework, the risk probability is finally defined as a random variable with distribution
\[ \Ri := \int_0^1 \d_{R \left( \a \right)} \text{ } \eta \left( \textrm{d}\a \right) .\]
This distribution of the risk is well-defined as soon as we set the distribution $\eta$ of the threshold $\alpha$.

%%%%%%%%%%%%%%%%%%%%%%%%%%%%%%%%%%%%%%%%%%%%%%%%%%%%%%%%%%%%%%%%%%%%%%%%%%%%%%%%%%%%%%%%%%%%%%%%%%%%%%%%%%%%%%%%%%%%%%%%%%%%%%%%%%%%%%%%%
% Uniform choice of the threshold
%%%%%%%%%%%%%%%%%%%%%%%%%%%%%%%%%%%%%%%%%%%%%%%%%%%%%%%%%%%%%%%%%%%%%%%%%%%%%%%%%%%%%%%%%%%%%%%%%%%%%%%%%%%%%%%%%%%%%%%%%%%%%%%%%%%%%%%%%
\subsection{Uniform distribution for the accident threshold}
\label{UniformDistribution}
In this sub-section, we discuss the choice of the distribution $\eta$ of the threshold $\a$.

A natural output in risk assessment is the average failure probability, $ E \left( \P \left( Y_x \in A \right)\right)$ in our case.
It is an estimator of the $P$-measure of the failure set based on the known response values $ \left( y_i = y \left( x_i \right) \right)_{1 \leq i\leq n} $.
If we now suppose the observation set $ \left( x_i \right)_{1 \leq i\leq n} $ covers $X$, then $ \Mi \left( x \right) $ is the characteristic function of $ \left\{ x \in X \mid y \left( x \right) \in A \right\} $ and consequently $ E \left( \P \left( Y_x \in A \right)\right) = P \left( y \left( x \right) \in A \right) $.
The estimator is therefore convergent if the factor space $X$ is appropriately sampled.
As we will see in the sequel, the causal model introduced in Section~\ref{FailureRisk} can be interpreted as an operator acting on the space of probability measures on $\left[0,1\right]$, which links the failure risk and the failure probability distributions.
Thus, it seems clever to keep the first moment of the latter when the transform is applied.
In that way, the mean failure risk remains a convergent estimator of the failure set measure.
This remark leads us to the first proposition.

If the model $\left(Y_x\right)$ is not available, then the failure probability is unknown.
The uniform (non-informative) prior appears intuitively as a reasonable choice here to represent ignorance, considering the principle of indifference.
In such situation, the failure risk distribution should also be uniform, meaning that the state of uncertainty has been preserved.
Based on this invariance property, we state the second proposition.

First of all, we introduce some definitions and notations, useful for the subsequent demonstrations.
As above, for each $ \a \in \left[ 0,1 \right]$, we set:
\[ R \left( \a \right) := P \left( \P \left( Y_x \in A \right) > \a \right) .\]
The function $R$ is the complementary cumulative distribution function (or tail distribution) of the random variable $\Mi \left( x \right) = \P \left( Y_x \in A \right) $, defined on the probability space $ \left( X,P \right) $.
Now, we can also look at $R$ as a random variable defined on the probability space $ \left( \left[ 0,1 \right],\eta \right)$, and we define, for all $ t \in \left[ 0,1 \right] $:
\[ G \left( t \right) := \eta \left( \left\{ \a \in \left[ 0, 1 \right] \mid R \left( \a \right) > t \right\} \right) = \eta \left( R \left( \a \right) > t \right) .\]
So $G$ is the tail distribution of the random variable $R$.

Moreover, we denote by $K$ the cumulative distribution function of $\eta$: $ K \left( u \right) = \eta \left( \left[ 0,u \right] \right)$ for all $u \in \left[ 0,1 \right] $.
Finally, let us recall the definition of the generalized inverse of the decreasing function $R$.

\begin{definition}[Generalized inverse]
Let $S$ a decreasing right continuous function on the interval $ \left[ 0,1 \right] $ with $ S \left( 0 \right) \leq 1 $ and $ S \left( 1 \right) = 0 $.
For all $ t \in \left[ 0,1 \right] $, we set:
\[ S^{-1} \left( t \right) = \sup \left\{ \a \in \left[ 0,1 \right] : S \left( \a \right) > t \right\}. \]
\end{definition}
At this stage, it is not difficult to check that $ G = K \circ R^{-1} $.
Let bethink us of this operator, which maps $G$ to $R$, in a more general setting.
The probability measure~$\eta$ is assumed fixed.
To any probability measure $m$ on the interval $ \left[ 0,1 \right] $, we can associate another probability measure~$m'$ on $ \left[ 0,1 \right] $ in the following way:
\begin{itemize}
\item firstly denote by $F_m$ the function $ F_m \left( \a \right) = m \left( (\a,1] \right) $,
\item secondly define $ F_{m'} = K \circ F_m^{-1} $.
\end{itemize}
Denote by $L_{\eta}$ the operator, acting on the space of probability measures on $ \left[ 0,1 \right] $, which associates $m'$ to $m$.

\begin{proposition}
${}$
\begin{enumerate}
\item The operator $ L_{\eta} $ preserves the first moment of the probability (i.e. $ \int_0^1 t \text{ } m \left( {\rm d} t \right) = \int_0^1 t \text{ } m' \left( {\rm d} t \right) $ for all $m$) if and only if the probability $\eta$ is uniform on $ \left[ 0,1 \right] $.
\item If the probability $\eta$ is uniform, then the operator $ L_{\eta} $ is its own inverse, meaning that $ \left( m' \right)' = m $.
\end{enumerate}
\end{proposition}

\begin{proof}[Proof of Proposition 1]
The first moment of the probability measure $m$ associated with $ F_m $ is equal to:
\[ \int_0^1 F_m \left( \a \right) \textrm{ d}\a .\]
The following neat identity can be proved:
\[ \int_0^1 F_m \left( \a \right) \textrm{ d}\a = \int_0^1 F_m^{-1} \left( \a \right) \textrm{ d}\a .\]
Thus, the first moment of $m$ is equal to the first moment of $m'$ if and only if:
\[ \int_0^1 K \left( F_m^{-1} \left( \a \right) \right) \textrm{ d} \a = \int_0^1 F_m^{-1} \left( \a \right) \textrm{ d}\a .\]
It can be easily verified that the identity map is the unique function $K$ which satisfies this equality for all choice of $m$.
This means that $\eta$ has to be a uniform distribution.
It is not difficult to check point 2.
\end{proof}

\begin{proposition}
The uniform distribution $\eta$ is the unique distribution for which the operator $ L_{\eta} $ applied to a uniform distribution gives a uniform distribution.
\end{proposition}
The proof does not present any difficulty.

In the sequel, $\eta$ is chosen as the uniform distribution.
At this stage of our presentation, the distribution of the risk probability is defined by
\[ \Ri := \int_0^1 \d_{R \left( \a \right)} \textrm{ d}\a .\]

%%%%%%%%%%%%%%%%%%%%%%%%%%%%%%%%%%%%%%%%%%%%%%%%%%%%%%%%%%%%%%%%%%%%%%%%%%%%%%%%%%%%%%%%%%%%%%%%%%%%%%%%%%%%%%%%%%%%%%%%%%%%%%%%%%%%%%%%%
% Monte Carlo approach
%%%%%%%%%%%%%%%%%%%%%%%%%%%%%%%%%%%%%%%%%%%%%%%%%%%%%%%%%%%%%%%%%%%%%%%%%%%%%%%%%%%%%%%%%%%%%%%%%%%%%%%%%%%%%%%%%%%%%%%%%%%%%%%%%%%%%%%%%
\subsection{Monte Carlo approach}
\label{MCapproach}
We describe in this sub-section how our reasoning changes according to practical possibilities.

As announced in Section~\ref{Intro_FailureRisk}, the quantity $ R \left( \a \right) $ has an integral formulation and cannot be expressed analytically in practice:
\[ R \left( \a \right) = \int_X \1_{\P \left( Y_x \in A \right) > \a} \text{ } P\left(\textrm{d}x\right) .\]
We use a MC method to get a numerical approximation of $ R \left( \a \right) $.
More precisely, an importance sampling is performed: independent samples $ \left( x_m \right)_{1\leq m\leq M} $, in the factor space $X$, are generated from the distribution $P$.
For more details on the related preferential MC method, see 4.5 p. 19 of \cite{Caflish}.
For each point $x_m$, the failure probability can be computed numerically since the Student distribution is tabulated.
Denote by $ n\left(\a\right) $ the number of MC draws $x_m$ such that $ \P \left( Y_{x_m} \in A \right) > \a $.
The MC estimator of $ R\left(\a\right) $ is the quotient $ n \left( \a \right)/M $.

Note that a classical Bayesian reasoning is used to model the uncertainty due to finite sampling.
In this approach, the MC sampling is considered as a binomial experiment.
This assumption is thoroughly justified for a good quality pseudo random number generator.
As a consequence, $R\left(\a\right)$ is a random variable whose distribution is the beta distribution $\beta_{\a}$ with shape parameters $ \left( n\left(\a\right) + 1, M - n\left(\a\right) + 1 \right) $, assuming a uniform prior.
Adding this additional probabilistic stratum in the model, we may appear a little bit pernickety.
Nevertheless, it guards us against the basic (but widespread) fallacy which consists in presuming that a system is perfectly safe as long as no accident has been observed.

In conclusion, taking into account the numerical aspects of the problem, we define the distribution of the risk probability as
\[ \Ri := \int_0^1 \beta_{\a} \textrm{ d}\a .\]

%%%%%%%%%%%%%%%%%%%%%%%%%%%%%%%%%%%%%%%%%%%%%%%%%%%%%%%%%%%%%%%%%%%%%%%%%%%%%%%%%%%%%%%%%%%%%%%%%%%%%%%%%%%%%%%%%%%%%%%%%%%%%%%%%%%%%%%%%
% Global strategy
%%%%%%%%%%%%%%%%%%%%%%%%%%%%%%%%%%%%%%%%%%%%%%%%%%%%%%%%%%%%%%%%%%%%%%%%%%%%%%%%%%%%%%%%%%%%%%%%%%%%%%%%%%%%%%%%%%%%%%%%%%%%%%%%%%%%%%%%%
\subsection{Global strategy}
\label{GlobalStrategy}
From a set of virtual experiments, the computer code emulator (Student field) is derived.
A stochastic simulation is then performed (uncertainty analysis) to extract the failure risk distribution.
Standard statistical quantities such as the mean, the standard deviation and a confidence interval are available to the decision-maker for risk assessment.
The confidence interval indicates the reliability of the estimate, the sources of uncertainty being the partial knowledge of the original function $y$ and the finite MC sampling.

If the prediction quality is too low, new data points have to be added to the observations set.
Let us describe succinctly a way to proceed.
Note first that the MC simulation only requires model evaluations.
Consequently, we have full scope to complete the observations set; it is in no way linked to the factors distribution $P$.
We could sample data uniformly in $X$ but, as each datum "costs the earth", it is strongly recommended to structure the data collection process.
This issue is addressed by experimental design techniques, which intend to optimize the information gathering.
We retain the differential entropy $ h \left( Y_x \right) := \E \left( - \log f \left( Y_x \right) \right) $, where $f$ is the density function of the random variable $Y_x$, as a measure of the lack of information (\cite{Cover}).
The information provided to the probabilistic model is increased if observations are made at locations in $X$ where the entropy is maximum.
The corresponding optimization problem is multi-modal: the objective is to find a set of local maxima for $ h \left( Y_x \right) $.
It can be solved by conventional methods such as gradient descent or metaheuristics as evolutionary algorithms.
We refer to \cite{Talbi} for a complete presentation of metaheuristics.
Once the observation candidates are identified, we are back to step one...
The iterative procedure is repeated until the precision requirements regarding the failure risk are met or the due date for the risk analysis is reached.

%%%%%%%%%%%%%%%%%%%%%%%%%%%%%%%%%%%%%%%%%%%%%%%%%%%%%%%%%%%%%%%%%%%%%%%%%%%%%%%%%%%%%%%%%%%%%%%%%%%%%%%%%%%%%%%%%%%%%%%%%%%%%%%%%%%%%%%%%
% Example and results
%%%%%%%%%%%%%%%%%%%%%%%%%%%%%%%%%%%%%%%%%%%%%%%%%%%%%%%%%%%%%%%%%%%%%%%%%%%%%%%%%%%%%%%%%%%%%%%%%%%%%%%%%%%%%%%%%%%%%%%%%%%%%%%%%%%%%%%%%
\subsection{Example and results}
\label{Examples}
In order to illustrate the interest of the risk assessment scheme presented in this article, three theoretical examples described in Appendix~\ref{AppendixExamples} are considered.
Using test functions with closed-form expressions, the failure probability can be calculated exactly.
Therefore, estimates can be compared to the true value.

The Gaussian mixture prior defined in Section~\ref{StudentField} is kitted out with the anisotropic $\g$-exponential correlation function (see Appendix~\ref{correlationfunction}) and trained using the MLE method described in Section~\ref{MLE}.
The resulting model has $D+1$ parameters, a reasonable level of complexity.
The factor distribution $P$ is assumed uniform.

As an alternative to our model-based Monte Carlo (MMC) method, we consider the brute-force Monte Carlo (BMC) method, recalled in Section~\ref{MCapproach}.
Let $k$ the number of defectives samples, i.e. such that $ y \left( x \right) \in A $, from $M$ trials.
The inferred failure risk follows the beta distribution with shape parameters $ \left( k+1,M-k+1 \right) $.
Note that this trivial model implicitly assumes a sampling of the factor space according to the distribution $P$ (uniform).
We used a quasi-random rather a pseudo-random source: $ \left( x_m \right)_{1 \leq m \leq M} $ are chosen as elements of the Sobol low discrepancy sequence.
Such sampling covers the factor space more evenly, a desirable property to obtain a relevant statistical population.
For the MMC and the BMC methods, the mean value and the boundaries of the smallest confidence interval, stated at the $90\%$ confidence level, are plotted as a function of the number $M$ of evaluations, $ 10\leq M\leq 1000 $.

Before going further, it is important to have in mind that a confidence interval only provides a statistical estimation of the error on the result.
It does not imply a strict condition of membership to the interval.
Indeed, whatever the sampling method used, this one may not catch the essential features of the function~$y$.
As a result, uncertainty can be underestimated, a dangerous situation in risk assessment.
This is particularly true for BMC, which does not take into account the spatial distribution of the data (geometrical structure of the space $X$).
This is the key difference with the MMC method we propose.
As a consequence, the comparison of the convergence properties of MMC and BMC should not focus on confidence intervals.

\subsubsection{Quadric example}
Let us start with the quadric example, considering a space $X$ of dimension $D=5$ and a true failure risk $ r_t = 1.008 \times 10^{-1} $ (the theoretical expression is available in the supplementary materials).
See Figure~\ref{figureQuadric} for a comparison of the convergence rates of the model-based and brute-force Monte Carlo.
It turns out that, in this case, MMC definitely outperforms BMC since it estimates the failure risk as accurate as BMC on 600 samples, requiring a third of the function evaluations.

The prior random process defines a distribution over a function set derived from the covariance function.
The Bayesian update selects a sub-set compatible with the observations.
Thus, the posterior is still random and the failure risk inherits this feature.
Although MMC has an extra source of uncertainty, its confidence interval is smaller.
At least 60 samples are necessary to get a confidence interval size inferior to 0.05.
This lower bound soars to 400 data points using BMC.
Besides, the MMC confidence interval always includes the failure risk true value.
All this suggests that the prior assumptions of the model are appropriate and significantly impact the prediction accuracy, which is not surprising, the function under study being quite simple.

\begin{figure}[h]
\includegraphics[width=3.2in,height=2.27in,angle=0]{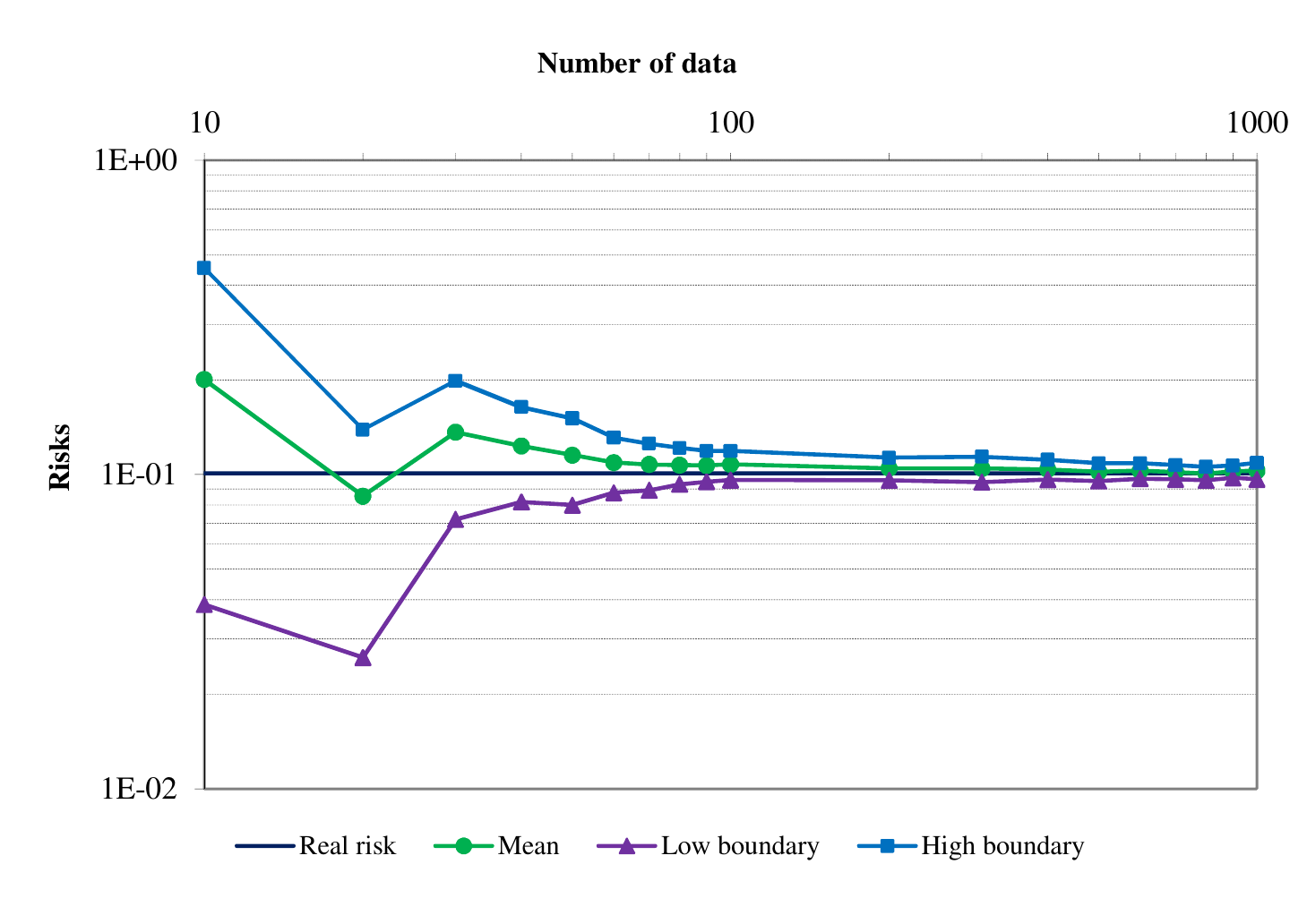}\hfill
\includegraphics[width=3.2in,height=2.27in,angle=0]{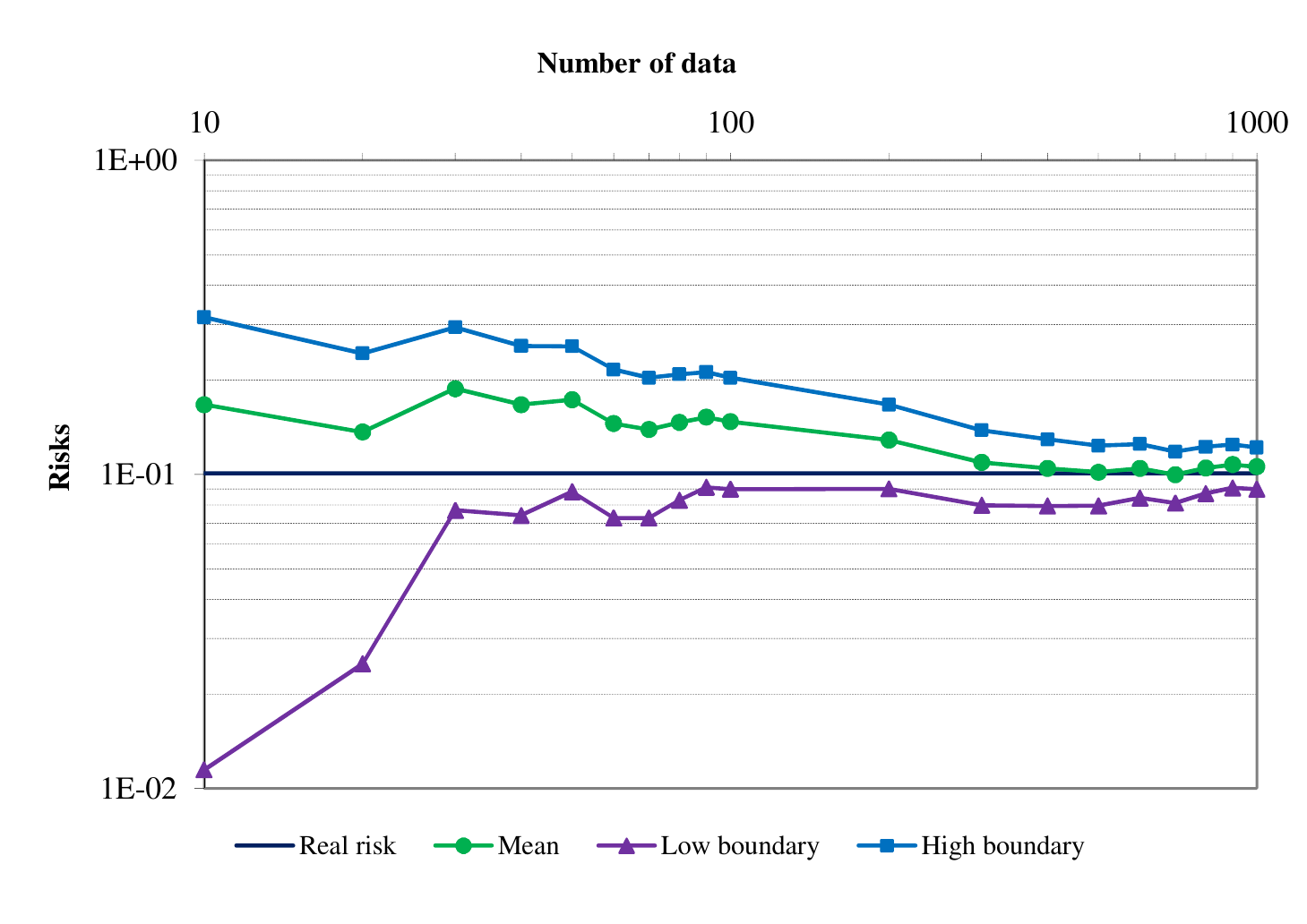}
\caption{Quadric example: convergence plots for MMC method (left) and BMC method (right)}
\label{figureQuadric}
\end{figure}

\subsubsection{Sine example}
The sine example ($D=3$) is a trickier problem since the function to be modeled is oscillating.
However, we expect the BMC method to perform correctly since the true failure risk is set to $ r_t = 2.048 \times 10^{-1} $, a relatively high value.
Indeed, the following rule of thumb regarding the sample size $M$ can be easily derived: $ M >> 1/r_t $.
Such a condition is easily fulfilled in this case.

Both methods give quite similar results, plotted in Figure~\ref{figureSine}.
The relative error on the mean failure risk predicted by MMC is less than $5\%$ beyond 300 samples, to be compared to 400 samples for the statistical mean.
The MMC and BMC confidence intervals size is less than 0.05 respectively beyond 600 and 800 samples.

We note that, below 200 samples, the MMC confidence interval is large, which indicates that the conditioning of the prior model does not restrict much the set of interpolating functions.
The propagation of information is therefore weak and we conclude that the $\gamma$-exponential correlation function does not really suit.
In addition, a spurious pinch of the confidence interval shows up for 30 samples, due to strong variations of the correlation function parameters.
The maximization of the marginal likelihood yields an overly confident estimate of the correlation function parameters posterior that would be obtained by classical Bayesian analysis.
Thus, MLE reaches its limits when information is scarce, firstly because it becomes very sensitive to the training set and secondly because the Bayesian posterior can be hardly approximated by its mode.

\begin{figure}[h]
\includegraphics[width=3.2in,height=2.27in,angle=0]{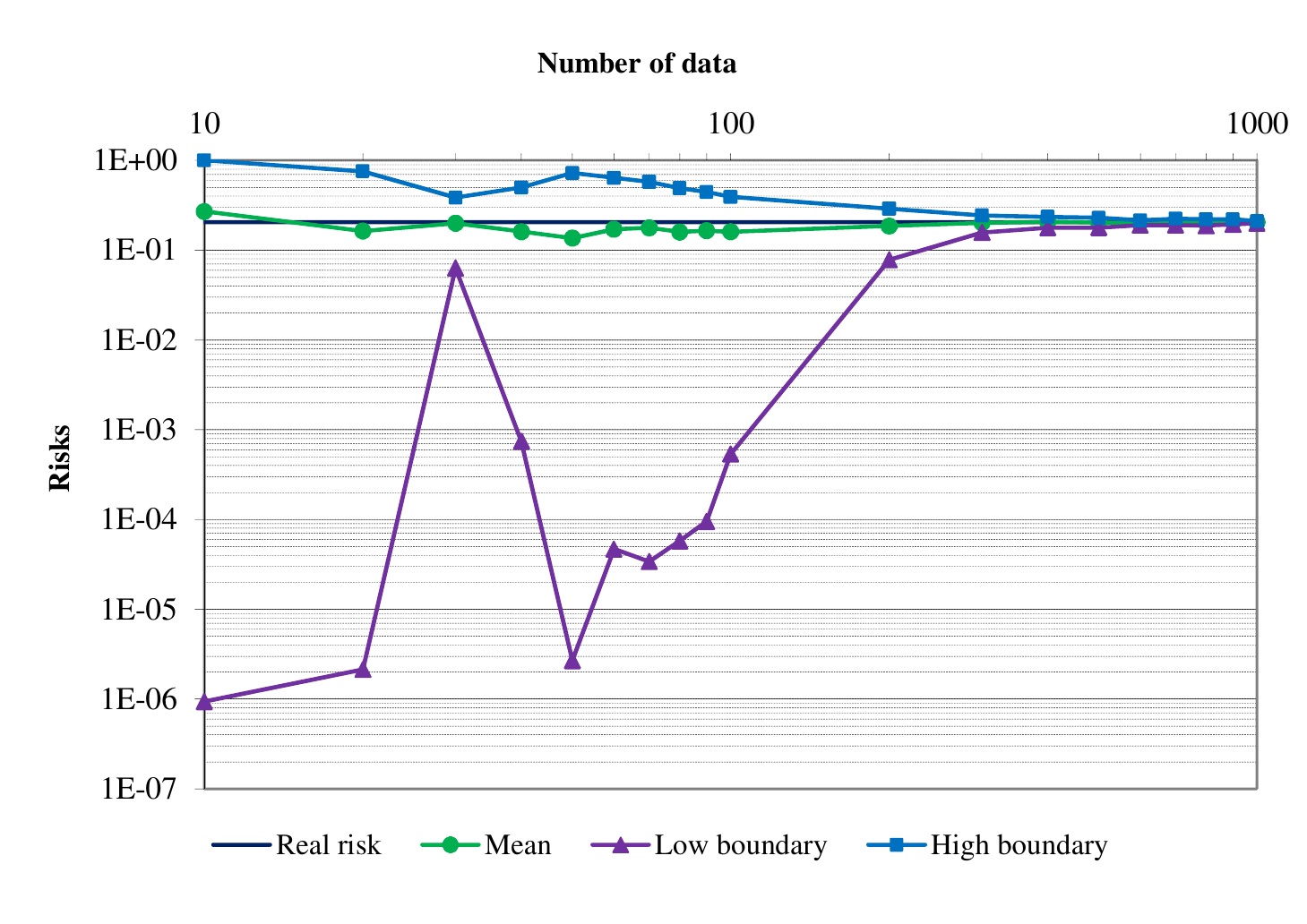}\hfill
\includegraphics[width=3.2in,height=2.27in,angle=0]{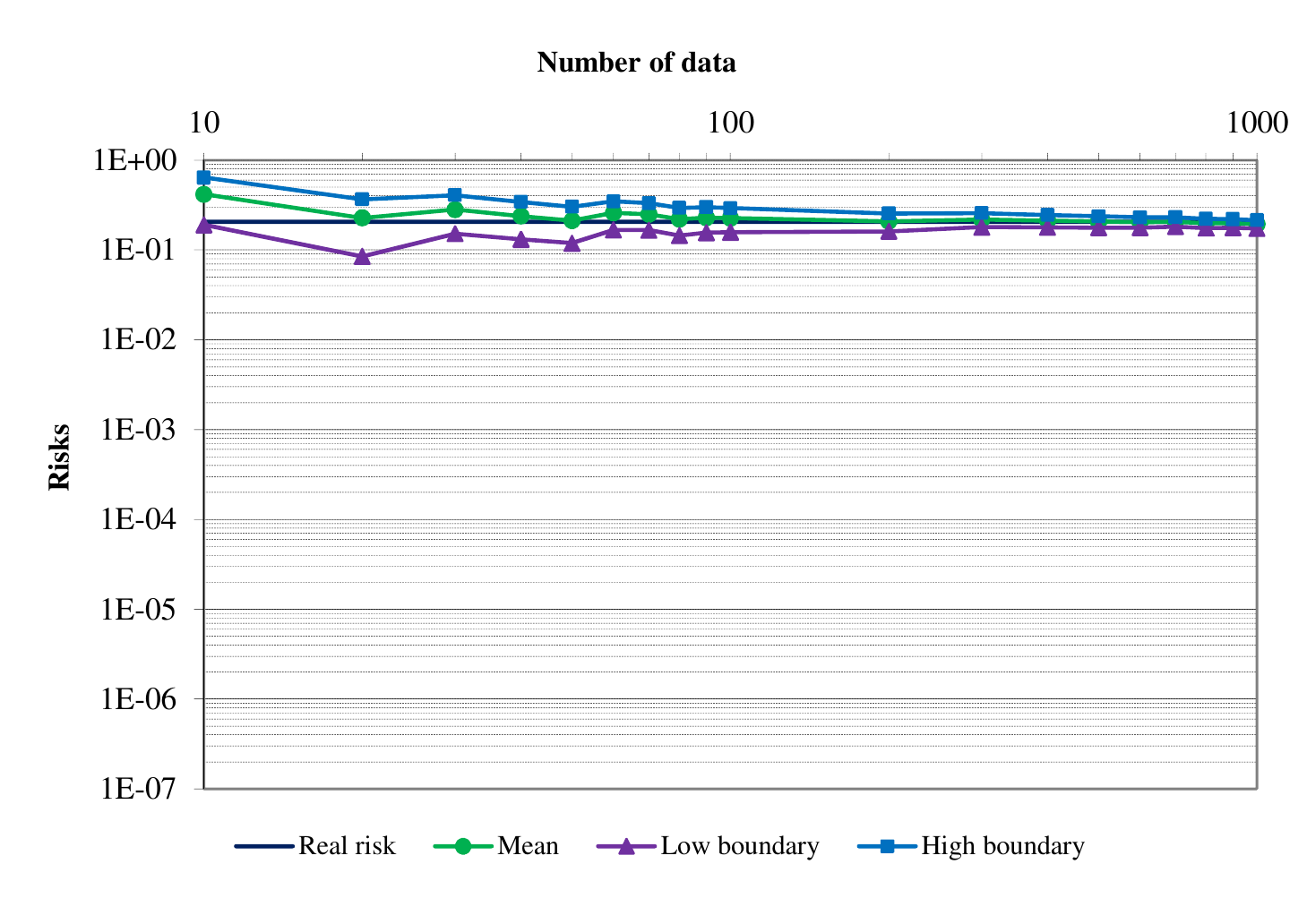}
\caption{Sine example: convergence plots for MMC method (left) and BMC method (right)}
\label{figureSine}
\end{figure}

\subsubsection{Bell-shaped example}
Consider the bell-shaped example ($D=5$).
The benchmark function, as a mixture of 10 multi-dimensional Gaussian functions, is hilly.
The failure set is disconnected: it is the union of 10 small hyperspheres of different radii.
The real failure risk is $ r_t = 4.341 \times 10^{-3} $.

Below 700 draws, the statistical mean decreases as $ O \left( n^{-1} \right) $, which is typical of an under-sampling of the failure set.
Indeed, the first defective part has been observed beyond 800 draws.
At the end of the sampling process ($M=1000$), the relative error on the mean failure risk is still 23\% for MMC and 31\% for BMC.
BMC is here obviously overconfident: the upper boundary of the 90\% confidence interval comes close to touching the risk true value.
MMC prediction is much more robust: the confidence interval safely flanks $r_t$ and the ratio of the mean value to the confidence interval size unambiguously indicates that additional data are required.

The raw analysis of the prediction error (see Section~\ref{ErrorStudy}) gives no evidence to conclude that information is lacking.
Indeed, $ \epsilon_r \leq 3.7\% $ with a probability of 0.9 and $\epsilon_r$ average is $1.7\%$, which is not exactly the signature of a poor model.
In contrast with global sensitivity analysis, discussed in \cite{Oakley}, goodness-of-fit requirements do not only depend upon the visiting probability distribution $P$: good accuracy is also required in this case at the failure set boundaries.
Thus, the decision regarding the model adequacy should not rest on a global criterion.
In addition, focusing on the model performances to validate the failure risk prediction may lead to over-quality: the accident set $A$ should also be taken into account.
The probabilistic risk assessment scheme we propose naturally overcome these difficulties.

\begin{figure}[h]
\includegraphics[width=3.2in,height=2.27in,angle=0]{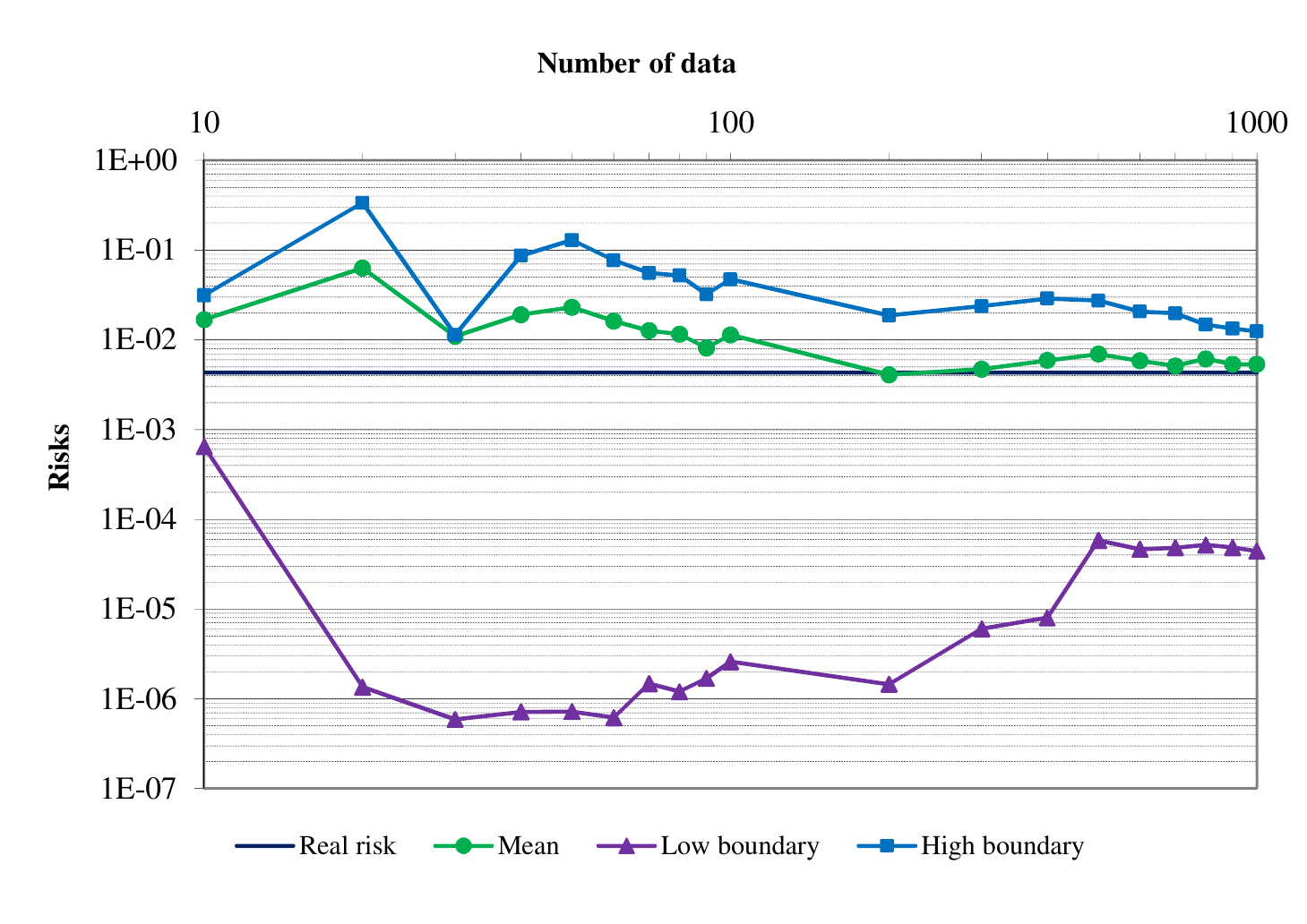}\hfill
\includegraphics[width=3.2in,height=2.27in,angle=0]{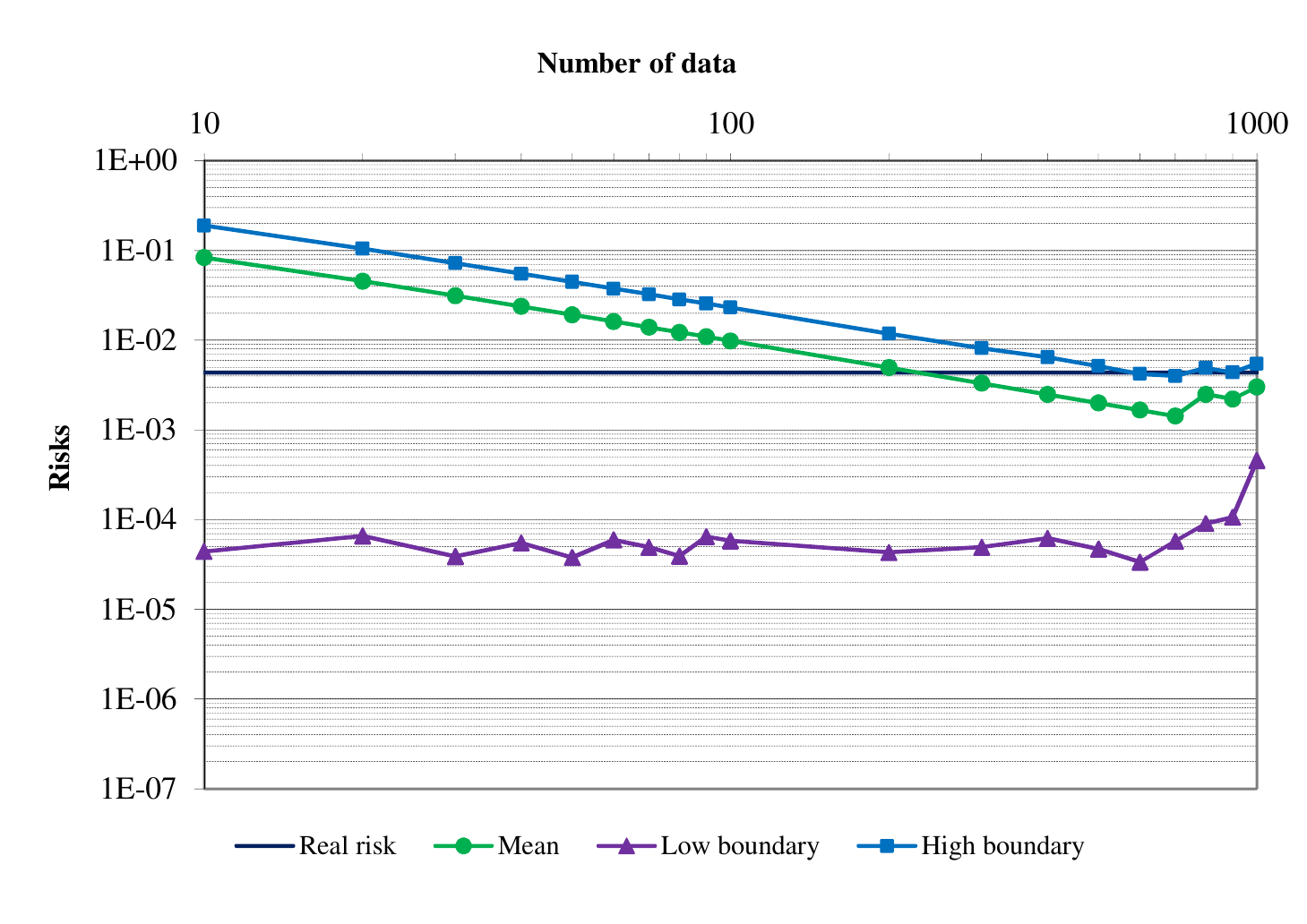}
\caption{Bell-shaped example: convergence plots for MMC method (left) and BMC method (right)}
\label{figureBell}
\end{figure}

%%%%%%%%%%%%%%%%%%%%%%%%%%%%%%%%%%%%%%%%%%%%%%%%%%%%%%%%%%%%%%%%%%%%%%%%%%%%%%%%%%%%%%%%%%%%%%%%%%%%%%%%%%%%%%%%%%%%%%%%%%%%%%%%%%%%%%%%%
% Conclusion
%%%%%%%%%%%%%%%%%%%%%%%%%%%%%%%%%%%%%%%%%%%%%%%%%%%%%%%%%%%%%%%%%%%%%%%%%%%%%%%%%%%%%%%%%%%%%%%%%%%%%%%%%%%%%%%%%%%%%%%%%%%%%%%%%%%%%%%%%
\section{Conclusion}
We introduce in this article a data sparing scheme in order to measure the robustness of a design to fabrication fluctuations.
It was evolved in order to offset the lack of information, either because of the complexity of the data acquisition process or due to the dimensionality of the state space.
The method is therefore particularly adapted to virtual experimentation, in which physically realistic computer experiments replace test runs to foresee product performances.
Indeed, simulations duration are often prohibitive; the numerical simulator is considered as a "black-box" function and approximated by an analytical emulator.
In this work, the regression is based on a Student process derived from low-informative priors (uniform) on the location and scale parameters of a Gaussian process.

Risk evaluation can be viewed as a binary classification problem, the product specifications defining the class membership.
Thus, we use the behavioral model as a probabilistic classifier, by interpreting predictive probabilities as degrees of belief.
In this scope, the failure risk naturally comes up as a random variable, randomness modeling the attitude of the decision-maker while exposed to uncertainty.
As a result, the failure risk distribution is a straightforward measure of the impact of the manufacturing process variability on the product performances, in a way that reflects the various sources of uncertainty: incomplete knowledge of the "black-box" function, values of the model variables, numerical approximations...

The relevance of the proposed methodology has been demonstrated on theoretical examples.
This study shows that, for risk assessment, model-based Monte Carlo provides a more reliable estimator of the failure risk than brute-force Monte Carlo.
Although it is not a crucial point in our argument, we noticed that introducing an intermediary analytical representation (i.e. incorporating prior knowledge) may come with a higher convergence rate of the MC sampling, thus reducing the computational effort for a given accuracy.

This work opens up new perspectives for the application of random field regression to engineering risk analysis and we can already give further lines of research:
\begin{itemize}
\item The natural extension of this work is the generalization of the approach in order to address the multi-responses problem.
A solution is to build a random field over a mixed state space including the factors as well as the responses.
The main drawback here is that the space dimension may be increased drastically.
In addition, it may be difficult to find an appropriate correlation function.
A possibly clever alternative is to model each response independently and to compute the upper bound of the failure probability of the union event.
The "worst case" risk probability is then easily derived, applying the procedure prescribed in this article.
\item We can study a more sophisticated regression model of the type
\[ \sum_{j=1}^k U_j f_j\left(x\right) + VW \]
(where $f_j$ are deterministic known functions).
It is the Student extension of some classical Gaussian models extensively described in the literature (see \cite{Santner}).
\item The following ternary classification problem can be investigated.
If the MC samples are separated into three groups:"good", "bad" and "undecidable", it is possible to define a deficiency risk.
Such random variable is a quantitative measure of the random model inability to detect defective parts.
It is intended to notify the decision-maker that the model is not informed enough to draw reliable conclusions.
\item In practice, we also have to control carefully numerical calculations and approximations.
A particularly resistant problem is the definite-positive matrix inversion when the matrix is closed to singular.
\end{itemize}

%%%%%%%%%%%%%%%%%%%%%%%%%%%%%%%%%%%%%%%%%%%%%%%%%%%%%%%%%%%%%%%%%%%%%%%%%%%%%%%%%%%%%%%%%%%%%%%%%%%%%%%%%%%%%%%%%%%%%%%%%%%%%%%%%%%%%%%%%
% Acknowledgments
%%%%%%%%%%%%%%%%%%%%%%%%%%%%%%%%%%%%%%%%%%%%%%%%%%%%%%%%%%%%%%%%%%%%%%%%%%%%%%%%%%%%%%%%%%%%%%%%%%%%%%%%%%%%%%%%%%%%%%%%%%%%%%%%%%%%%%%%%
\if0\blind
{
\section*{Acknowledgments}
This research is developed through a PhD thesis, supported by STMicroelectronics and the \textit{Association Nationale de la Recherche et de la Technologie} (http://www.anrt.asso.fr/) via a CIFRE funding.
} \fi

%%%%%%%%%%%%%%%%%%%%%%%%%%%%%%%%%%%%%%%%%%%%%%%%%%%%%%%%%%%%%%%%%%%%%%%%%%%%%%%%%%%%%%%%%%%%%%%%%%%%%%%%%%%%%%%%%%%%%%%%%%%%%%%%%%%%%%%%%
% Appendix: examples
%%%%%%%%%%%%%%%%%%%%%%%%%%%%%%%%%%%%%%%%%%%%%%%%%%%%%%%%%%%%%%%%%%%%%%%%%%%%%%%%%%%%%%%%%%%%%%%%%%%%%%%%%%%%%%%%%%%%%%%%%%%%%%%%%%%%%%%%%
\appendix
 \section{Appendix: examples}
 \label{AppendixExamples}
Discussions in Sections~\ref{ErrorStudy} and~\ref{Examples} are based on three arbitrary chosen benchmark functions that we describe briefly here and more precisely in the supplementary materials.
In each of these examples, the out of specification space is defined by $ A := \left[ m, + \infty \right)$ for $m\in \R$.

%%%%%%%%%%%%%%%%%%%%%%%%%%%%%%%%%%%%%%%%%%%%%%%%%%%%%%%%%%%%%%%%%%%%%%%%%%%%%%%%%%%%%%%%%%%%%%
% Quadric example
%%%%%%%%%%%%%%%%%%%%%%%%%%%%%%%%%%%%%%%%%%%%%%%%%%%%%%%%%%%%%%%%%%%%%%%%%%%%%%%%%%%%%%%%%%%%%%
\subsection{Quadric example}
\label{quadric_example}
Fix $m > 0$, $ a_1,a_2,\ldots,a_D > m^2 $, $ X := \left[-1,1 \right]^D $ and for $ \b x := \left( x_1, x_2, \ldots, x_D \right) \in X $:
\[ y \left( \b x \right) := \sqrt{\sum_{i=1}^D a_i x_i^2 } .\]

%%%%%%%%%%%%%%%%%%%%%%%%%%%%%%%%%%%%%%%%%%%%%%%%%%%%%%%%%%%%%%%%%%%%%%%%%%%%%%%%%%%%%%%%%%%%%%
% Sine example
%%%%%%%%%%%%%%%%%%%%%%%%%%%%%%%%%%%%%%%%%%%%%%%%%%%%%%%%%%%%%%%%%%%%%%%%%%%%%%%%%%%%%%%%%%%%%%
\subsection{Sine example}
Fix $ a_1,a_2,\ldots,a_D $ non zero integers, $ X:= \left[ 0,1 \right]^D $ and for $ \b x := \left( x_1, x_2, \ldots, x_D \right) \in X $:
\[ y \left( \b x \right) := \sin \left( 2 \pi \sum_{i=1}^D {a_i x_i} \right) .\]

%%%%%%%%%%%%%%%%%%%%%%%%%%%%%%%%%%%%%%%%%%%%%%%%%%%%%%%%%%%%%%%%%%%%%%%%%%%%%%%%%%%%%%%%%%%%%%
% Bell-shaped example
%%%%%%%%%%%%%%%%%%%%%%%%%%%%%%%%%%%%%%%%%%%%%%%%%%%%%%%%%%%%%%%%%%%%%%%%%%%%%%%%%%%%%%%%%%%%%%
\subsection{Bell-shaped example}
Let $R$ be a positive integer, $ X = \left[ 0,1 \right]^D$ and for $ \b x \in X $ :
\[ y \left( \b x \right) := \max_{1 \leq i \leq R} f_i \left( \b x, \b \mu_i, \s_i \right) \] 
where for all $ 1 \leq i \leq R $, $ \b \mu_i \in \R^D $, $\sigma_i \in \R$, $\sigma_i > 0$ and
\[ f_i \left( \b x, \b \mu_i, \s_i \right) = \frac1{ \left( \sqrt{2\pi} \s_i \right)^D} \exp \left( - \frac1{2\s_i^2} \left\| \b x - \b \mu_i \right\|^2 \right) \]
such that for the fixed threshold $m$, the failure areas associated to each of the $f_i$ do not overlap each other.

%%%%%%%%%%%%%%%%%%%%%%%%%%%%%%%%%%%%%%%%%%%%%%%%%%%%%%%%%%%%%%%%%%%%%%%%%%%%%%%%%%%%%%%%%%%%%%
% Appendix: correlation function
%%%%%%%%%%%%%%%%%%%%%%%%%%%%%%%%%%%%%%%%%%%%%%%%%%%%%%%%%%%%%%%%%%%%%%%%%%%%%%%%%%%%%%%%%%%%%%
\section{Appendix: correlation function}
\label{correlationfunction}
The correlation function chosen in Sections~\ref{ErrorStudy} and~\ref{Examples} belongs to $ \g$-exponential family:
\[ \rho\left( x, x' \right) := \exp \left( - \left( \sum_{i=1}^D \left( \frac{ x_i - x'_i }{l_i} \right)^2 \right)^{ \frac{\g}2} \right) \quad \textrm{for } x, x' \in X .\]
It is stationary and anisotropic.
Roughly speaking, the exponent $ \g \in \left(0, 2 \right] $ controls the smoothness of the random process (it is mean square differentiable only when $\gamma=2$) and the positive real numbers $ l_1, l_2, \ldots, l_D $ are characteristic length-scales, defining the influence hyper-ellipsoid of an observation point.

%%%%%%%%%%%%%%%%%%%%%%%%%%%%%%%%%%%%%%%%%%%%%%%%%%%%%%%%%%%%%%%%%%%%%%%%%%%%%%%%%%%%%%%%%%%%%%
% Supplementary material
%%%%%%%%%%%%%%%%%%%%%%%%%%%%%%%%%%%%%%%%%%%%%%%%%%%%%%%%%%%%%%%%%%%%%%%%%%%%%%%%%%%%%%%%%%%%%%
\section*{Supplementary material}
\begin{description}
\item[Theorem] Proof of the Theorem 3 (.pdf)
\item[Examples] Full numerical description of the examples. (.pdf)
\end{description}

%%%%%%%%%%%%%%%%%%%%%%%%%%%%%%%%%%%%%%%%%%%%%%%%%%%%%%%%%%%%%%%%%%%%%%%%%%%%%%%%%%%%%%%%%%%%%%%%%%%%%%%%%%%%%%%%%%%%%%%%%%%%%%%%%%%%%%%%%
% Bibliography
%%%%%%%%%%%%%%%%%%%%%%%%%%%%%%%%%%%%%%%%%%%%%%%%%%%%%%%%%%%%%%%%%%%%%%%%%%%%%%%%%%%%%%%%%%%%%%%%%%%%%%%%%%%%%%%%%%%%%%%%%%%%%%%%%%%%%%%%%
\bibliographystyle{apalike}
\bibliography{Article}

\end{document}

\maketitle

We describe now three theoretical examples used in the article.
These dummy problems have been designed to meet the following requirements.
First, it is defined for an arbitrary factor space dimension $D$.
In each of these examples, the space $X$ is a hyper-rectangle in $\R^D$ of the form $ X = \prod_{i=1}^D \left[ X_{\min,i}, X_{\max,i} \right]$.
It is equipped with a probability measure $P$.
Secondly, a theoretical risk, denoted $r_t$, can be expressed in closed-form.
Let $m\in \R$, the out of specifications space is defined by $ A := \left[ m, +\infty \right)$.
As a consequence, the real failure set is
$ X_A := \left\{ x \in X \mid y \left( x \right) > m \right\} $.
The true failure risk is given by $ r_t = P \left( X_A \right)$.
If the distribution $P$ is uniform, $r_t$ simply equals the volume of $X_A$.
All numerical calculations are performed with Matlab (2012.a).

%%%%%%%%%%%%%%%%%%%%%%%%%%%%%%%%%%%%%%%%%%%%%%%%%%%%%%%%%%%%%%%%%%%%%%%%%%%%%%%%%%%%%%%%%%%%%%
% Quadric example
%%%%%%%%%%%%%%%%%%%%%%%%%%%%%%%%%%%%%%%%%%%%%%%%%%%%%%%%%%%%%%%%%%%%%%%%%%%%%%%%%%%%%%%%%%%%%%
\section{Quadric example}
Fix $m > 0$, $ a_1,a_2,\ldots,a_D > m^2 $, $ X := \left[-1,1 \right]^D $ and for $ \b x := \left( x_1, x_2, \ldots, x_D \right) \in X $:
\[ y \left( \b x \right) := \sqrt{\sum_{i=1}^D a_i x_i^2 } .\]
This is a degenerate quadric.
In particular, for $D=2$, it's a cone.
The boundary of failure set $X_A$ is a generalization of an ellipse in dimension $D$.
The failure set volume is:
\[ \frac1{\text{vol} \left( X \right)} \int_{X} 1_{ \left[ y \left( \b x \right) > m \right]} \text{ d} \b x = 1 - \frac1{2^D} \frac{\pi^{\frac{D}2} m^D}{ \Gamma \left( \frac{D}2 +1 \right)} \prod_{i=1}^D a_i^{-\frac12} .\]
The particular example that we consider in the article is defined as follow.

Set $D$ to $5$, $ \left( a_1 = 4.0 , a_2 = 4.2 , a_3 = 4.4 , a_4 = 4.6 , a_5 = 4.8 \right) $ and $m$ to $1.9$.
The theoretical risk is then $ r_t = 1.008 \times 10^{-1} $.

%%%%%%%%%%%%%%%%%%%%%%%%%%%%%%%%%%%%%%%%%%%%%%%%%%%%%%%%%%%%%%%%%%%%%%%%%%%%%%%%%%%%%%%%%%%%%%
% Sine example
%%%%%%%%%%%%%%%%%%%%%%%%%%%%%%%%%%%%%%%%%%%%%%%%%%%%%%%%%%%%%%%%%%%%%%%%%%%%%%%%%%%%%%%%%%%%%%
\section{Sine example}
Fix $ a_1,a_2,\ldots,a_D $ non zero integers, $ X:= \left[ 0,1 \right]^D $ and for $ \b x := \left( x_1, x_2, \ldots, x_D \right) \in X$:
\[ y \left( \b x \right) := \sin \left( 2 \pi \sum_{i=1}^D {a_i x_i} \right) .\]
It can be proved that for all $ m \in \left[ -1,1 \right] $:
\[ \int_X 1_{ \left[ y \left( \b x \right) > m \right]} \text{ d}\b x = \int_0^1 1_{ \left[ \sin \left(2 \pi t \right) > m \right]} \text{ d}t = \frac12 - \frac{\arcsin \left( m \right)}{\pi} .\]
Here is the particular example that we consider.
Set $D$ to 3, $ \left( a_1 = 1 , a_2 = -1 , a_3 = 2 \right) $ and $m$ to 0.8.
The theoretical risk is $ r_t = 2.048 \times 10^{-1} $.

%%%%%%%%%%%%%%%%%%%%%%%%%%%%%%%%%%%%%%%%%%%%%%%%%%%%%%%%%%%%%%%%%%%%%%%%%%%%%%%%%%%%%%%%%%%%%%
% Bell-shaped example
%%%%%%%%%%%%%%%%%%%%%%%%%%%%%%%%%%%%%%%%%%%%%%%%%%%%%%%%%%%%%%%%%%%%%%%%%%%%%%%%%%%%%%%%%%%%%%
\section{Bell-shaped example}
Let $R$ be a positive integer, $ X = \left[ 0,1 \right]^D$ and for $ \b x \in X$:
\[ y \left( \b x \right) := \max_{1 \leq i \leq R} f_i \left( \b x , \b \mu_i, \s_i \right) \] 
where for all $ 1 \leq i \leq R $, $ \b \mu_i \in \R^D $, $\sigma_i \in\R $, $\sigma_i > 0$ and
\[ f_i \left( \b x, \b \mu_i, \s_i \right) = \frac1{ \left( \sqrt{2\pi} \s_i \right)^D} \exp \left( - \frac1{2\s_i^2} \left\| \b x - \b \mu_i \right\|^2 \right) \]
such that for a fixed threshold $m$, the failure areas associated to each $f_i$ do not overlap each other.
When this condition is satisfied, it follows that:
\[ \int_X 1_{ \left[ y \left( \b x \right) > m \right]} \text{ d}\b x = \sum_{i=1}^R \int_X 1_{ \left[ f_i \left( \b x, \b \mu_i, \s_i \right) > m \right] } \text{ d}\b x .\]
More precisely, in order to get a simple closed-form expression for $r_t$, we consider the case defined by the statements below, involving the quantities:
\[ r_i := \sqrt{ - 2 \s_i^2 \log \left( m \left( 2 \pi \right)^{\frac{D}2} \s_i^D \right)}, \quad \forall 1 \leq i \leq R \]
\begin{enumerate}
 \item Failure subsets are not empty: 
 \[ 0 < \s_i < \frac1{\sqrt{2\pi}m^{\frac1{D}}}, \quad \forall 1 \leq i \leq R \]
 \item Failure subsets are entirely included in the space $X$:
 \[ 0 \leq \mu_{ij} \pm r_i \leq 1, \quad \forall 1 \leq i \leq R, \forall 1 \leq j \leq D \]
 \item Failure subsets are disjoint: 
 \[ \left\| \b \mu_i - \b \mu_j \right\| > r_i + r_j ,\quad \forall j \neq i .\]
\end{enumerate}
Then, for each $ 1 \leq i \leq R $, the elementary failure area is a $D$-sphere centered in $ \b \mu_i$ and with the radius $r_i$.
Its volume is:
\[ \frac{\pi^{\frac{D}2} r_i^D}{\Gamma \left( \frac{D}2 + 1 \right)} .\]
Finally, the failure risk is simply the sum of the volumes of $R$ disjoint $D$-spheres:
\[ \int_X 1_{ \left[ y \left( \b x \right) > m \right]} \text{ d}\b x = \frac{\pi^{\frac{D}2}}{\Gamma\left(\frac{D}2 + 1 \right)} \sum_{i=1}^R r_i^D .\]
Here is the particular example presented in the article.
Set $D$ to 5, $R$ to 10 and $m$ to 0.1.
The bell-shaped function $y$ consists of 10 multi-dimensional Gaussian functions defined in Table~\ref{BellShapedTab}.

The theoretical risk is about $ r_t = 4.341 \times 10^{-3} $.

\begin{table}[h]
\begin{center}
\begin{tabular}{|c|c|c|c|c|c|c|}
\hline
& $ \mu_{i1} $ & $ \mu_{i2} $ & $ \mu_{i3} $ & $ \mu_{i4} $ & $ \mu_{i5} $ & $\s_i$ \\
\hline
$f_1$ & $4.889\times 10^{-1}$ & $6.241\times 10^{-1}$ & $6.791\times 10^{-1}$ & $3.955\times 10^{-1}$ & $3.674\times 10^{-1}$ & $6.259\times 10^{-1}$ \\
\hline
$f_2$ & $5.870\times 10^{-1}$ & $4.139\times 10^{-1}$ & $3.091\times 10^{-1}$ & $2.638\times 10^{-1}$ & $7.588\times 10^{-1}$ & $6.297\times 10^{-1}$ \\
\hline
$f_3$ & $7.637\times 10^{-1}$ & $5.588\times 10^{-1}$ & $1.838\times 10^{-1}$ & $4.980\times 10^{-1}$ & $5.179\times 10^{-1}$ & $6.292\times 10^{-1}$ \\
\hline
$f_4$ & $2.318\times 10^{-1}$ & $3.963\times 10^{-1}$ & $7.051\times 10^{-1}$ & $5.586\times 10^{-1}$ & $7.566\times 10^{-1}$ & $6.299\times 10^{-1}$ \\
\hline
$f_5$ & $5.806\times 10^{-1}$ & $2.989\times 10^{-1}$ & $7.137\times 10^{-1}$ & $3.605\times 10^{-1}$ & $7.185\times 10^{-1}$ & $6.298\times 10^{-1}$ \\
\hline
$f_6$ & $3.805\times 10^{-1}$ & $3.022\times 10^{-1}$ & $2.945\times 10^{-1}$ & $2.598\times 10^{-1}$ & $4.635\times 10^{-1}$ & $6.258\times 10^{-1}$ \\
\hline
$f_7$ & $1.684\times 10^{-1}$ & $7.401\times 10^{-1}$ & $8.728\times 10^{-1}$ & $3.555\times 10^{-1}$ & $1.788\times 10^{-1}$ & $6.304\times 10^{-1}$ \\
\hline
$f_8$ & $7.885\times 10^{-1}$ & $5.891\times 10^{-1}$ & $4.363\times 10^{-1}$ & $5.497\times 10^{-1}$ & $3.530\times 10^{-1}$ & $6.300\times 10^{-1}$ \\
\hline
$f_9$ & $8.601\times 10^{-1}$ & $3.204\times 10^{-1}$ & $6.636\times 10^{-1}$ & $9.056\times 10^{-1}$ & $6.854\times 10^{-1}$ & $6.318\times 10^{-1}$ \\
\hline
$f_{10}$ & $2.995\times 10^{-1}$ & $8.727\times 10^{-1}$ & $3.751\times 10^{-1}$ & $7.984\times 10^{-1}$ & $8.168\times 10^{-1}$ & $6.318\times 10^{-1}$ \\
\hline
\end{tabular}
\end{center}
\caption{Parameters values of the bell-shaped function.}
\label{BellShapedTab}
\end{table}